\documentclass[12pt]{amsart}
\usepackage{amscd,amssymb,graphics,lineno}
\usepackage{mathrsfs} 

\newtheorem{theorem}{Theorem}[section]
\newtheorem{corollary}[theorem]{Corollary} 
\newtheorem{lemma}[theorem]{Lemma}
\newtheorem{proposition}[theorem]{Proposition}
\theoremstyle{definition}
\newtheorem{definition}[theorem]{Definition}

\newtheorem{problem}{Open problem}
\newtheorem{question}[problem]{Question}

\theoremstyle{remark}

\numberwithin{equation}{section}

\def\norm#1{\left\Vert#1\right\Vert}

\def\R{{\mathbb R}}

\def\Sym{{\mathrm{Sym}\,}}

\def\e{\varepsilon}

\def\Aut{{\mbox{\rm Aut}\,}}

\newcounter{quest}
\stepcounter{quest}

\def\a{\alpha}
\def\d{\delta}
\def\ga{\gamma}
\def\Ga{\Gamma}
\def\k{\kappa}

\def\Si{\Sigma}
\def\stm{\setminus}
\def\o{\omega}
\def\sB{\mathcal B}
\def\sE{\mathcal E}
\def\sF{\mathcal F}
\def\sG{\mathcal G}
\def\sK{\mathcal K}
\def\sN{\mathcal N}
\def\sP{\mathcal P}
\def\sS{\mathcal S}
\def\sT{\mathcal T}
\def\sU{\mathcal U}
\def\sW{\mathcal W}
\def\sbs{\subset}

\def\obr{^{-1}}

\def\fC{\mathfrak C}
\def\fV{\mathfrak V}
\def\fW{\mathfrak W}
\def\rar{\rightarrow}
\def\lar{\leftarrow}

\def\Card{{\mathrm{Card}\,}} 
\def\card#1{\left|\,#1\,\right|}

\def\Arh{Arhan\-gel'\-ski\u\i}

\begin{document}

\title[Sequentially continuous homomorphisms] 
{Real-valued measurable cardinals and sequentially continuous
homomorphisms} 

\dedicatory{Dedicated to the memory of Mitrofan Choban}

%
%
%
%

\author{Vladimir V. Uspenskij}

\address{Department of Mathematics, 321 Morton Hall, Ohio
University, Athens, Ohio 45701, USA}

\email{uspenski@ohio.edu}



\date{August 22, 2021}

\keywords{Variety of groups, sequentially continuous, $f$-sequential space,
refinements of group topologies, locally compact groups}

 \begin{abstract} 
A.~V.~Arkhangel'ski\u\i\ asked in 1981 if the variety $\fV$ of topological
groups
generated by free topological groups on metrizable spaces coincides with the
class of all topological groups. We show that if there exists a real-valued
measurable cardinal then the variety $\fV$ is a proper subclass of the class of
all topological groups. A topological group
$G$ is called $g$-sequential if for any topological group $H$
any sequentially continuous homomorphism
$G\rar H$ is continuous. We introduce the concept of a $g$-sequential
cardinal and prove that a locally compact group is $g$-sequential if and
only if its local weight is not a $g$-sequential cardinal.
The product of a family of non-trivial $g$-sequential topological groups is
$g$-sequential if and only if the cardinal of this family is not
$g$-sequential. Suppose $G$ is either the unitary group of a Hilbert space
or the group of all self-homeomorphisms of a Tikhonov cube. Then $G$ is
$g$-sequential if and only if its weight is not a $g$-sequential
cardinal. Every compact group of Ulam-measurable cardinality
admits a strictly finer countably compact group topology.
 \end{abstract}

\thanks {{\it 2020 Mathematics Subject Classification.}
Primary 22A05; secondary 54C08, 54E35, 22B05, 03E55.
\newline
The author was supported by a Humboldt Research Fellowship}


\maketitle

\section{Introduction}

In 1981 A.V.Arkhangel'ski\u\i\ \cite[p.~171, Remark~f]{A1} asked if
\begin{quote}
(A) {\it 
the family of free topological groups
on metrizable spaces
generates (by means of the operations of taking products, subgroups, and
factor groups) the class of all topological groups.}
\end{quote}

Assuming the existence of real-valued measurable cardinals, we answer this
question in the negative (Theorem~\ref{2.13}). Actually it suffices to assume the
existence of smaller cardinals, which we call {\it g-sequential\/}
(Definition~\ref{1.3}).

\begin{definition} \cite{Mo1, Mo2, Mo3, Mo4}. A class of topological groups is
called a
{\it variety\/} if it is closed under arbitrary products, subgroups, and
topological quotient groups.
\end{definition}

Our results do not depend on whether we assume or not that all topological
groups are Hausdorff. Unless otherwise stated, we assume that all groups are
Hausdorff, but in Section~2 it will be more convenient to consider also
non-Hausdorff group topologies.
Throughout the paper we denote by $\fV$ the
variety of topological groups generated by free topological groups on
metrizable spaces. Thus (A) is the assertion that $\fV$ coincides with the
class of all topological groups. Recall that the
{\it (Markov) free topological group\/} $F(X)$ on a Tikhonov space $X$ is
characterized by the following property: $X$ is a subspace of $F(X)$, and
every continuous mapping $f$ from $X$ to a topological group $G$ extends
uniquely to a continuous homomorphism $\bar f: F(X)\rar G$.

The present paper was inspired by the manuscript \cite{MNPS} by S.~Morris,
P.~Nickolas, V. ~Pestov and S.~Svetlichny, where an approach
was suggested to a positive solution of Arkhangel'ski\u\i's question. We
discuss the ideas of \cite{MNPS} below in Section~4. It follows from the
results of \cite{MNPS} and from our arguments that if there are no sequential
cardinals \cite {ACh}
and if the assertion (F) below is true,  
then Arkhangel'ski\u\i's question has a positive answer.
Moreover, in this case every topological group is a quotient of a
closed subgroup of the free topological group on a metrizable space, as
conjectured in \cite{MNPS}.

\smallskip
(F) {\it If $X\sbs Y$ are Tikhonov spaces and the fine uniformity $\sU_Y$ on $Y$
induces the fine uniformity $\sU_X$ on $X$, then the natural injection $F(X)\to F(Y)$
between free topological groups is a topological embedding.}
\smallskip

This assertion is contained in \cite{Sip, S2}. However, the arguments in \cite{Sip, S2} 
are hard to follow, and some specialists doubt that the proof of (F) is complete. We therefore
prefer a cautious approach and consider the statement 

\smallskip
(M) {\it if there are no $g$-sequential cardinals, then (A) holds}
\smallskip

\noindent
as a conditional result that is true if (F) is.

Another way to prove (M) could be first to establish the following:

\smallskip
(U) {\it For every Tikhonov space $X$ the free topological group $F(X)$ admits a topologically
faithful unitary representation, that is, is isomorphic to a subgroup of the unitary group $U(H)$
of a (non-separable) Hilbert space.}

Here $U(H)$ is equipped with the strong operator topology that it inherits from the product $H^H$.
The conjecture (U) can be viewed as a version of the following problem due to A.\,Kechris:
is every Polish group a quotient of a closed subgroup of the unitary group 
of a separable Hilbert space? Indeed, it was noted in \cite[Proposition~4.2]{U5} that this question
has a positive answer if the free group on the space of irrationals admits a topologically faithful
unitary representation. If (U) is true, then so is (M) (this follows from Theorem~\ref{t:4.11} below).
It was proved in \cite{U5} that for every
Tikhonov space $X$ the free
Abelian group $A(X)$ admits a topologically faithful unitary representation. It follows that if there
are no $g$-sequential cardinals, all Abelian topological groups belong to the variety $\fV$.

An ultrafilter on a set $A$ is {\it $\kappa$-complete\/} if it is closed
under intersections of families of cardinality $<\kappa$.
 Let us say that a cardinal $\Card(A)$ is
{\it Ulam-measurable\/} if there exists
an $\o_1$-complete free ultrafilter on $A$. 
(A {\em free} ultrafilter is the same as a nonprincipal ultrafilter.)
A cardinal is Ulam-measurable
if and only if it greater than or equal to the first measurable cardinal.
As usual \cite[Definition 10.3]{Jech}, an uncountable cardinal $\kappa$
is measurable if there exists
a $\kappa$-complete free ultrafilter on $\kappa$.


The variety $\fV$ consists of quotients of subgroups of products of free
topological groups on metrizable spaces. The following definition,
introduced
in \cite{MNPS}, plays a crucial role in the study of this variety.
\begin{definition}[\cite{MNPS}] 
\label{1.2}
A topological group $G$ is
{\it g-sequential\/} if one of the following three equivalent properties
holds:
\begin{enumerate}
\item for any topological group $H$, any sequentially continuous
homomorphism $f:G\rar H$ is continuous;
\item $G$ admits no strictly finer group topology with the
same convergent sequences;
\item $G$ is isomorphic to a quotient group of the free topological group
of a metrizable space.
\end{enumerate}
\end{definition}

The equivalence of the conditions 1 and 2
is clear, and their equivalence to the condition 3
is Theorem~3.7 in \cite{MNPS} (see also Section~4 below).
It follows that the variety~$\fV$ can be described
as the variety generated by all $g$-sequential groups. Hence the
following assertion~(G) implies that \Arh's question has a positive
answer, that is, the variety~$\fV$ coincides with the class of all
topological groups.

\smallskip
(G) {\it Every topological group is isomorphic to a subgroup of a
$g$-sequential topological group.}
\smallskip

Our main result, Theorem~2.13, implies that (G) is incompatible with the
existence of real-valued measurable cardinals, so one cannot expect to
prove (G) in $ZFC$. It is not clear if (G) is consistent.
Assuming there are no $g$-sequential cardinals, we prove that 
for every Hilbert space $H$ the unitary group $U(H)$ is $g$-sequential, 
and that
the same is true for the group $\Aut I^\tau$ of all self-homeomorphisms of
a Tikhonov cube $I^\tau$ (Theorem~4.11). Every topological group with a
countable base is isomorphic to a subgroup of $\Aut I^\omega$ \cite{U1}, but
it is an open problem if a similar assertion holds for uncountable
cardinals, that is, if every group of weight $\tau$ is isomorphic to a
subgroup of $\Aut I^\tau$. If it is true, then Theorem~4.11 implies that (G)
holds under the assumption that there are no $g$-sequential cardinals.

We say that the {\em local weight} of a topological group  $G$ is $\le \k$ 
if the neutral element of $G$ has a neighborhood of weight $\le \k$.
If $G$ is a locally compact group of non-Ulam-measurable local weight, then
every sequentially continuous homomorphism $f:G\rar H$ to a {\it locally
compact} group $H$ is continuous \cite{V}. It is natural to ask what
locally compact groups are $g$-sequential, in other words, what locally
compact groups do not admit strictly finer group topologies with the same
convergent sequences. Some results in this direction were obtained
in \cite{CR2}, \cite{AJ}. 
We prove that a
locally compact group $G$ is $g$-sequential if and only if its
local weight is not $g$-sequential (Theorem~3.14).

A slight modification of our main construction of refinements of group
topologies yields the following result (Theorem~3.16):
every locally compact group $G$ of non-Ulam-measurable local weight admits
a strictly finer group topology which agrees with the original one on
every set of non-Ulam-measurable cardinality. If $G$ is compact in the original
topology, it follows that in the
new topology it is countably compact (moreover, $\kappa$-compact for every
non-Ulam-measurable cardinal $\kappa$). By a theorem of A.~V.~\Arh\
\cite{A2}, a necessary condition for a compact group $G$ to admit a strictly
finer countably compact group topology is that the weight of $G$ be
an Ulam-measurable cardinal. Theorem~3.16 shows that this condition is also
sufficient. Under the assumption that $G$ is either Abelian or connected
this was proved in \cite{CR2}.

The second part of Theorem~3.16 says that the result of \cite{V} mentioned
above can be reversed: if $G$ is a locally compact group of Ulam-measurable
local weight, then there exists a sequentially continuous discontinuous
homomorphism $f:G\rar H$ of $G$ to a locally compact group $H$.

Theorem~3.14 implies that $g$-sequential cardinals
can be defined as follows:
a cardinal $\tau$ is {\it $g$-sequential\/} if the
compact group $2^\tau$ is not $g$-sequential (here 2 denotes a discrete
group consisting of two points).
We accept another property as the definition.
For a set $E$ let $P(E)$ denote the set of all subsets of $E$.
\begin{definition} 
\label{1.3}
A {\it subadditive measure\/} on a set $E$ is a function
$p: P(E)\rar [0,1]$ such that:
\begin{enumerate}
\item $p(A\cup B)\le p(A)+p(B)$ for every $A,B\in P(E)$ with
$A\cap B=\emptyset$;
\item for any decreasing sequence $A_1\supset A_2\supset\dots\in P(E)$
with $\bigcap A_n=\emptyset$ we have $\lim p(A_n)=0$.
\end{enumerate}
A cardinal $\tau=\Card(E)$ is {\it g-sequential} if there is
a subadditive measure $p$ on $E$ such that
$p(E)=1$ and $p(F)=0$ for all finite subsets $F\sbs E$.
\end{definition}

Every subadditive measure is monotone (that is, $A\sbs B$ implies
$p(A)\le p(B)$), so the assumption that $A$ and $B$ are disjoint can be
omitted in the condition 1.
A subadditive measure $p$ on $E$ such that $p(A\cup B)=p(A)+p(B)$ for every
$A,B\in P(E)$ with $A\cap B=\emptyset$ is a $\sigma$-additive measure in
the usual sense. A cardinal $\Card(E)$ is {\it real-valued
measurable\/} if there is a $\sigma$-additive
measure $p: P(E)\rar [0,1]$ such that
$p(E)=1$ and $p(F)=0$ for all finite subsets $F\sbs E$.
It follows that real-valued measurable cardinals are $g$-sequential.
\begin{definition} A topological space $X$ is {\it f-sequential\/} if every
sequentially continuous real-valued function $f:X\rar \mathbb R$ is continuous.
\end{definition}

This concept was considered in \cite{MNPS} under another name. Recall that
a cardinal $\tau$ is {\it sequential\/} \cite{ACh} if the space
$2^\tau$ is not $f$-sequential. Every $f$-sequential group is
$g$-sequential. Applying this remark to the group $2^\tau$, we see that
every $g$-sequential cardinal is sequential. Another way to prove this is to
note that every subadditive measure $p:P(E)\rar [0,1]$ is sequentially
continuous on $P(E)$, which is identified with the compact space $2^E$
(Proposition~2.3). If $p(E)=1$ and $p(F)=0$ for all finite
subsets $F\sbs E$, then $p$ is not continuous.

Under Martin's Axiom, for any cardinal `Ulam-measurable' is equivalent to
`sequential' and to `real-valued measurable' \cite[theorem 2.2]{ACh}
and hence also to `$g$-sequential'.
It is not clear if it can be proved without additional set-theoretic
assumptions that $g$-sequential cardinals coincide with either real-valued
measurable cardinals or with sequential cardinals.

In Section~2 we use the concept of a subadditive measure to introduce the
main tool for the proof of Theorem~2.13: for any topological space
$(X, \sT)$ we define a canonical refinement $\sT_g$ of the topology $\sT$
such that convergent sequences in $(X, \sT_g)$ remain the same as in
$(X, \sT)$. We then prove Theorem~2.13: if $\Card(E)$ is a $g$-sequential
cardinal, the topological group $\Sym(E)$ of all permutations of $E$ does
not
belong to the variety $\fV$, hence the existence of $g$-sequential cardinals
implies that the variety $\fV$ is proper. In Section~3 we study refinements
of locally compact group topologies and prove Theorems~3.14 and ~3.16
mentioned above. We also prove that
the product of a family of $g$-sequential groups
is $g$-sequential provided that the cardinality of the family is not
$g$-sequential (Theorem~3.2).
In Section~4 we discuss, following \cite{MNPS}, a possible way to prove
\Arh's conjecture~(A) under some additional assumptions. The arguments of
\cite{MNPS} depend on the fact that some products of Banach spaces are
$f$-sequential. It is known that the product of a family $\{X_\a: \a\in A\}$
of separable metric spaces is $f$-sequential if $\Card(A)$ is non-sequential
\cite[Theorem~1.5]{ACh}. We prove (Theorem~4.4) that this remains true for
arbitrary metric (or, more generally, bi-sequential) spaces. We then prove
Theorem~4.11, which answers the question of when the groups $U(H)$ and
$\Aut I^\tau$ are $g$-sequential and reduces the conjecture~(G) to the
problem whether the group $\Aut I^\tau$ is a universal topological group of
weight~$\tau$.

\section
{If there are large cardinals, the variety $\fV$ is proper}

For any topological space $(X,\sT)$ we define in a canonical way a
refinement $\sT_g$ of the topology $\sT$ so that convergent sequences in $X$
are the same for $\sT$ and for $\sT_g$ (Definition~\ref{d:2.1}).

Let $\sS=\{U_\a: \a \in A\}$ be a family of open sets in $X$, and let $p$
be a subadditive measure on the index set $A$ (Definition~1.3). For any
$\e>0$ let $W(\sS,p,\e)$ be the set of all $x\in X$ such that the set
$\{\a\in A: x\notin U_\a\}$ has $p$-measure $<\e$. If $A(\e)$ denotes the
collection of all $B\sbs A$ with $p(A\setminus B)<\e$, then
$$
W(\sS,p,\e)=\bigcup_{B\in A(\e)} \bigcap_{\a\in B} U_\a.
$$
\begin{definition}
\label{d:2.1}
For a topological space $(X,\sT)$ let $\sT_g$ be the topology
generated by the collection $\sB$ of the sets $W(\sS,p,\e)$ for all possible
choices of $\sS$, $p$ and $\e$.  We say that
$\sT_g$ is the {\it g-modification\/} of $\sT$, and denote by $X_g$ the
space $X$ equipped with the topology $\sT_g$.
\end{definition}

It is easy to see that $\sB$ is actually a base for $\sT_g$. A similar
construction was used in \cite{DW}.
\begin{proposition} 
\label{p:2.2}
(a) The operation of $g$-modification preserves
continuous maps: if $f:X\rar Y$ is continuous, then $f:X_g\rar Y_g$ is
also continuous;
(b) the operation of $g$-modification is compatible with subspaces: if
$Y$ is a subspace of $X$, then $Y_g$ is a subspace of $X_g$.
\end{proposition}

The proof is straightforward.\qed

\smallskip
Before we prove that spaces $X$ and $X_g$ have the same convergent
sequences,
let us note that subadditive measures are nothing else as sequentially
continuous seminorms on groups of the form $2^E$. Recall that a non-negative
real-valued function on a group $G$ is a {\it seminorm} if $p(e)=0$ (here
$e$ is the neutral element) and $p(xy^{-1})\le p(x)+p(y)$ for all
$x,y\in G$. Identifying $P(E)$, the set of subsets of $E$, with the compact
group $2^E$, we see that every sequentially continuous seminorm $p:2^E\rar
[0,1]$ is a subadditive measure on $E$. Conversely, let $p$ be a subadditive
measure on $E$. Then $p$ is a seminorm on the compact group $2^E$. To prove
that $p$ is sequentially continuous, it suffices to check that $p$ is
sequentially continuous at the unity, or that $\lim p(A_n)=0$ for
any sequence $\{A_n\}$ of subsets of $E$ which converges to the empty set.
The last assumption means that $\bigcap B_n=\emptyset$, where $B_n=\bigcup_
{k\ge n} A_k$, so the definition of a subadditive measure implies that
$\lim p(B_n)=0$. Since $p(A_n)\le p(B_n)$, it follows that $\lim p(A_n)=0$.
We thus have established
\begin{proposition} 
\label{p:2.3}
A function $p:P(E)\rar [0,1]$ is a subadditive
measure on $E$ in the sense of Definition~1.3 if and only if $p$ is a
sequentially continuous seminorm on the compact group~$2^E$.
\qed\end{proposition}

\begin{proposition} 
\label{p:2.4}
The topology $\sT_g$ is finer than $\sT$.
Convergent sequences in $X$ for the topologies $\sT$ and
$\sT_g$ are the same, so the identity map $X\rar X_g$ is sequentially
continuous.
\end{proposition}

\begin{proof}
If $\sS=\{U_\a: \a \in A\}\sbs \sT$ and $p$ is an atomic measure of full
mass~1 concentrated at a point $\a\in A$, then $W(\sS,p,\e)=U_\a$. It
follows that $\sT_g$ is finer than $\sT$. Suppose that a sequence $\{x_n\}$
converges to $x$ in $(X, \sT)$. We must show that every neighbourhood $W$ of
$x$ of the form $W=W(\sS,p,\e)$, where $\sS=\{U_\a:\a\in A\}$ and $p$ is a
subadditive measure on $A$, contains all but finitely many $x_n$'s. Let
$A_n$ be the set of all $\a\in A$ such that $x_k\notin U_\a$ for some
$k>n$, and let $B$ be the set of all $\a\in A$ such that $x\notin U_\a$.
The sequence of $A_n$'s is decreasing and $\bigcap A_n\sbs B$. Since $x\in
W $, the definition of the set $W(\sS,p,\e)$ implies that
$p(B)<\e$. Since $p$ is sequentially continuous (Proposition~1.2), we have
$\lim p(A_n) = p(\bigcap A_n) \le p(B) <\e$. Pick $N$ so that $p(A_N)<\e$.
Then $x_k\in W$ for every $k>N$.
\end{proof}
\begin{proposition} 
\label{p:2.5}
The operation of $g$-modification preserves
finite products:\newline
$(X\times Y)_g=X_g\times Y_g$ for any spaces $X$ and $Y$.
\end{proposition}

\begin{proof}
Proposition~2.2(a) implies that the projections $(X\times Y)_g
\rar X_g$ and $(X\times Y)_g\rar Y_g$ are continuous, hence the topology
of $(X\times Y)_g$ is finer than that of $X_g\times Y_g$. To prove the
converse, suppose that $z=(x,y)\in X\times Y$, and let $W=W(\sS,p,\e)$ be a
neighbourhood of $z$ in $(X\times Y)_g$. We must find $U$ open in $X_g$ and
$V$ open in $Y_g$ such that $z\in U\times V\sbs W$. Let $\sS=\{O_\a: \a\in
A\}$. Without loss of generality we may assume that $z\in \bigcap \sS$.
Otherwise let $B$ be the set of $\a\in A$ such that $z\notin O_\a$. Since
$z\in W$, we have $p(B)<\e$. Let $\e'=\e-p(B)$, let $A'=A\setminus B$, and
let $p'$ be the restriction of the subadditive measure $p$ to $A'$. Then for
$\sS'=\{O_\a: \a\in A'\}$ we have $z\in \bigcap \sS'\sbs W(\sS', p', \e')
\sbs W$. So assume that $z\in \bigcap \sS$. For every $\a\in A$ pick
open sets $U_\a\sbs X$ and $V_\a\sbs Y$ so that $z\in U_\a\times V_\a\sbs
O_\a$. Let $\sP_1=\{U_\a: \a\in A\}$ and $\sP_2=\{V_\a: \a\in A\}$. Let
$U=W(\sP_1, p, \e/2)$ and $V=W(\sP_2, p, \e/2)$. Then $U$ is open in $X_g$,
$V$ is open in $Y_g$ and $z\in U\times V\sbs W$.
\end{proof}
\begin{proposition} 
\label{p:2.6} 
The operation of $g$-modification preserves
topological groups: if $X$ is a topological group, then so is $X_g$.
\end{proposition}

\begin{proof} We must show that the multiplication $X_g\times X_g\rar X_g$
and the inversion $X_g\rar X_g$ are continuous. This follows from
Propositions~2.2(a) and~2.5.\end{proof}

Plainly the same argument can be applied to any other class of topological
algebras.
\begin{definition}
\label{d:2.7}
A topological space $(X, \sT)$ is {\it g-stable} if
$\sT=\sT_g$.
\end{definition}

It follows from Propositions~2.4 and~2.6 that $g$-sequential groups are
$g$-stable. Any subspace of a $g$-stable space is $g$-stable (Proposition~
2.2(b)).
\begin{proposition} 
\label{p:2.8} 
The property of being $g$-stable is preserved
by quotient maps.
\end{proposition}

\begin{proof} Let $f:X\rar Y$ be a quotient map. Then $Y$ admits no strictly
finer topology for which $f$ remains continuous. If $X$ is $g$-stable, that
is $X=X_g$, then $f:X\rar Y_g$ is continuous (Proposition 2.2(a)).
It follows that $Y=Y_g$.\end{proof}

Let $\fC$ be the smallest class of topological spaces which contains all
metrizable spaces (or just a convergent sequence) and is closed under
arbitrary disjoint sums, finite products and quotient maps. M.\,~Hu\v sek
asked if the class $\fC$ contains all topological spaces. This problem is
investigated in \cite{DW}, where it is proved that if there is a real-valued
measurable cardinal, then the answer is no. Since $g$-stable spaces are
plainly preserved by sums, it follows from Propositions~2.5 and~2.8 that
all spaces in $\fC$ are $g$-stable. If there is a $g$-sequential cardinal,
then not every space is $g$-stable (and vice versa). Hence in the result
of Dow and Watson quoted above `real-valued measurable' can be replaced
by `$g$-sequential'.

A subset $A$ of a partially ordered set $(L,\le)$ is {\it directed\/} if
for any $a,b\in A$ there is $c\in A$ with $a\le c$ and $b\le c$.
\begin{definition}
\label{d:2.9}
Let $L_1,\ L_2$ be complete lattices. An increasing map
$h:L_1\rar L_2$ is {\it lattice-continuous} if for any directed subset
$\Ga\sbs L_1$ we have $h(\sup \Ga) = \sup \{h(\ga): \ga\in \Ga\}$.
\end{definition}

For a group $G$ let $L_G$ be set of all (not necessarily Hausdorff) group
topologies on $G$. The set $L_G$ is a complete lattice. For an onto
homomorphism $f:G\rar H$ let $f_*:L_G\rar L_H$ be the map which assigns to
each $\sT\in L_G$ the quotient topology of $\sT$ on $H$.
\begin{proposition} 
\label{p:2.10}
For any onto homomorphism $f:G\rar H$ of groups
the map $f_*:L_G\rar L_H$ is lattice-continuous.
\end{proposition}

\begin{proof} Let $\Ga\sbs L_G$ be a directed set of group topologies on $G$.
Let $\sT=f_*(\sup \Ga)$ and $\sT'=\sup \{f_*(\sP):\sP\in \Ga\}$. We must
show that $\sT=\sT'$. Since $f_*$ is increasing, it is clear that $\sT'\sbs
\sT$. To prove the reverse inclusion, it suffices to show that every
$\sT$-neigh\-bour\-hood $V$ of the unity $e_H$ in $H$ is a
$\sT'$-neigh\-bour\-hood. Let $U=f^{-1} (V)$. Then $U$ is a neighbourhood
of the unity $e_G$ in $G$ relative to the topology $\sup\Ga$ and
hence, since $\Ga$ is directed, also relative to some $\sP\in \Ga$.
The map $f:(G,\sP)\rar (H,f_*(\sP))$, being a quotient homomorphism of
topological groups, is open. It follows that $V=f(U)$ is a neighbourhood of
$e_H$ relative to $f_*(\sP)$ and hence also relative to $\sT'$.\end{proof}
\begin{proposition} 
\label{p:2.11}
 Let $\fW$ be the class of all topological
groups $(G, \sT)$ with the following property:

The topology $\sT$ can be represented as the least upper bound of a set of
$g$-stable (not necessarily Hausdorff) group topologies on $G$.

\noindent Then $\fW$ is a variety of topological groups containing all
$g$-sequential topological groups.
\end{proposition} 

\begin{proof} It is clear that the class $\fW$ is closed under arbitrary
products and subgroups. The remark after Definition~2.7 shows that all
$g$-sequential groups are in $\fW$. It remains to prove that the class $\fW$
is closed under quotients. Note that the least upper bound $\sT=\sup(\sT_1,
\sT_2)$ of two $g$-stable topologies on a set $X$ is $g$-stable. This
follows from Propositions~2.2(b) and~2.5, since the space $(X, \sT)$ is
homeomorphic to a subspace of the product $(X,\sT_1)\times (X, \sT_2)$.
Consequently, a topological group $(G,\sT)$ is in $\fW$ if and only if
the topology $\sT$ can be represented as the least upper bound of a
{\it directed\/} set of $g$-stable group topologies on $G$. In virtue of
Propositions~2.8 and~2.10 the last property is preserved by quotient
homomorphisms.\end{proof}

For a set
$E$ we denote by $\Sym(E)$ the topological group of all permutations of $E$,
equipped with the topology of pointwise convergence ($E$ being considered
as a discrete space). For any $F\sbs E$ let $H_F$ be the group of all
permutations of $E$ which leave fixed every point in $F$. The family
$\{H_F: F\sbs E,\ F\ \text{finite}\}$ is a base at the unity of
$\Sym(E)$.
If $F$ is a singleton $\{x\}$, we write $H_x$ instead of $H_{\{x\}}$.
\begin{proposition} 
\label{2.12}
If $\Card(E)$ is a $g$-sequential cardinal,
the topological group $\Sym(E)$ does not belong to the variety $\fW$
defined in Proposition~2.11.
\end{proposition}

\begin{proof} (a) We first show that the group $G=\Sym(E)$ is not
$g$-stable. Let
$p$ be a subadditive measure on $E$ which witnesses that $\Card(E)$ is
$g$-sequential, that is, $p(E)=1$ and $p(F)=0$ for every finite subset $F
\sbs E$. For every $x\in E$ let $H_x$, as above, be the stability
subgroup at $x$.
Denote the family $\{H_x: x\in E\}$ by $\sS$. The set $W=W(\sS, p, 1)$,
defined as in the beginning of this Section, consists of all $f\in G$ with
the following property: the set of points $x\in E$ which are moved by $f$
has $p$-measure $< 1$. The set $W$ is not a neighbourhood of the unity
in $G$. Indeed, let $F$ be a finite subset of $E$. Pick $f\in G$ so that
$F$ is the set of all $f$-fixed points in $E$. Then $f\notin W$, since $f$
moves every point in $E\setminus F$ and $p(E\setminus F)=1$. Thus $f\in
H_F\setminus W$, so $H_F$ is not a subset of $W$. Since the subgroups
$H_F$ form a base at the unity in $G$, it follows that $W$ is not open
in $G$. On the other hand, $W$ is open in $G_g$ (Definition~2.1).
Thus $G$ is not $g$-stable.

(b) The proof of Proposition~2.11 shows that a topological group is in the
variety $\fW$ if and only if its topology can be written as $\sup \Ga$ for
some directed family $\Ga$ of $g$-stable group topologies. In virtue of the
part (a) of the proof, the topology $\sT$ of $G$ is not $g$-stable. Hence
to prove that $G\notin \fW$ it suffices to show that for any directed
family $\Ga$ of group topologies on $G$ such that $\sT=\sup \Ga$ we have
$\sT\in \Ga$. Pick a point $x\in E$. Since $\Ga$ is directed, there is
a topology $\sP\in \Ga$ such that the subgroup $H_x$ is a
$\sP$-neighbourhood of the identity. Since the subgroup $H_y$ is conjugate
to $H_x$ for every $y\in E$ and since $\sP$ is a group topology, $H_y$ is
a $\sP$-neighbourhood of the identity for every $y\in E$, and the same is
true for every subgroup $H_F=\bigcap\{H_y: y\in F\}$, where $F$ is a finite
subset of $E$. It follows that $\sT=\sP$. \end{proof}

Let $\fW$ be the variety defined in Proposition~2.11. We noted in
Section~1 that the variety $\fV$ is generated by $g$-sequential groups,
hence Proposition~2.11 implies that $\fV\sbs\fW$. It follows from
Proposition~2.12 that
the group $\Sym(E)$ is not in $\fV$ if $\Card(E)$ is $g$-sequential.
We have thus arrived at our main result:
\begin{theorem} 
\label{2.13}
Assume there exist $g$-sequential cardinals. Then
the variety $\fV$ generated by free topological groups of metrizable spaces
is a proper subclass of the class of all topological groups. For example,
if $\Card(E)$ is a $g$-sequential cardinal, the topological group
$\Sym(E)$ does not belong to $\fV$.
\end{theorem}

Salvador Hernandez 
raised
the following question: is $g$-modification an idempotent operation? In other words,
is it true that $(X_g)_g=X_g$ for every space $X$? Equivalently, is it true that $X_g$
is $g$-stable? The answer turns out to be positive (Proposition~\ref{p:Salva}). 

We need some preparations.
Suppose $B=\bigcup_{\a\in A} B_\a$ is a disjoint union, $p$ is a subadditive measure on $A$,
and for every $\a\in A$ a subadditive measure $q_\a$ on $B_\a$ is given.
Then one can define a subadditive measure
$\mu=\bigvee_pq_\a$ on $B$ as follows: if $E\sbs B$, then
$$
\mu(E)=\inf\{\e>0: p(\{\a\in A: q_\a(E\cap B_\a)\ge \e\})\le\e\}.
$$
Note that $\mu(E)<\e$ implies 
$p(\{\a\in A: q_\a(E\cap B_\a)\ge \e\})<\e$ implies $\mu(E)\le\e$.
Let us check that $\mu$ is a subadditive measure. Suppose $E_1, E_2\sbs B$ and $\mu(E_i)<\e_i$,
$i=1,2$. Let $A_i=\{\a\in A: q_\a(E_i\cap B_\a)\ge \e_i\}$. Then $p(A_i)<\e_i$. If $\a\in A$
is such that $q_\a((E_1\cup E_2)\cap B_\a)\ge\e_1+\e_2$, then $\a\in A_1\cup A_2$, hence
$$
 p(\{\a\in A:q_\a((E_1\cup E_2)\cap B_\a)\ge\e_1+\e_2\})\le p(A_1)+p(A_2)<\e_1+\e_2.
$$
It follows that $\mu(E_1\cup E_2)\le\e_1+\e_2$.

We also have to check that $\mu(E_n)\to 0$ for every decreasing sequence $(E_n)$ of subsets of $B$
such that $\bigcap E_n=\emptyset$. Let $\e>0$ be given. Put 
$$
A_n=\{\a\in A:q_\a(E_n\cap B_\a)\ge\e\}.
$$
The sets $A_n$ decrease and have empty intersection, therefore $p(A_n)<\e$ for $n$ large enough.
For such $n$ we have $\mu(E_n)\le \e$.

Now suppose that $f_\a>0$ is given for every $\a\in A$. Let $\e>0$.
We claim that 
\begin{equation}
\exists \d>0\,\forall E\sbs B\, (\mu(E)<\d 
\implies 
p\{\a: q_\a(E\cap B_\a)\ge f_\a\}<\e).
\label{eq:existsdelta}
\end{equation}
Indeed, let $A_\d=\{\a:f_\a<\d\}$. Then $p(A_\d)\to 0$ as $\d\to 0$. Pick $\d<\e/2$ so that
$p(A_\d)<\e/2$. We claim that $\d$ has the required property.
Let $E\sbs B$ be such that $\mu(E)<\d$. We must show that the set 
$C=\{\a: q_\a(E\cap B_\a)\ge f_\a\}$ has $p$-measure $<\e$.
Let $C_1=C\cap A_\d$ and $C_2=C\stm C_1$. Then $C_2\sbs \{\a:q_\a(E\cap B_\a)\ge\d\}$,
and the latter set has $p$-measure $<\d$ since $\mu(E)<\d$. Thus $C=C_1\cup C_2$ 
is covered by two sets of $p$-measure $<\e/2$. It follows that $p(C)<\e$.

\begin{proposition}
\label{p:Salva}
The $g$-modification is an idempotent operation: $(X_g)_g=X_g$ for every $X$. In other words,
$X_g$ is $g$-stable. 
\end{proposition}

\begin{proof}
Let $x_0\in X$. A basic open set in $(X_g)_g$ that contains $x_0$ has the form $W(\sS,p,\e)$, where 
$\sS=\{U_\a: \a \in A\}$ is a family of open sets in $X_g$, and $p$
is a subadditive measure on the index set $A$. We may assume that each $U_\a$
is a basic open set of the form $U_\a=W(\sG_\a,q_\a,f_\a)$ for some family 
$\sG_\a=\{G_\beta:\beta\in B_\a\}$ of open sets in $X$, where $q_\a$ is a subadditive measure
on the index set $B_\a$ and $f_\a>0$. We also may assume that $x_0\in G_\beta$ for all indices
$\beta$ (discard those $G_\beta$ for which this does not hold).
Consider the subadditive measure $\mu=\bigvee_pq_\a$ on $B=\bigcup B_\a$
constructed above, and pick $\d>0$ as in Eq.~(\ref{eq:existsdelta}). We claim that 
$$
x_0\in W(\sG, \mu, \d)\sbs W(\sS,p,\e),
$$
where $\sG=\{G_\beta:\beta\in B\}$. This will prove that the topology of $X_g$ is finer than
(and hence equal to) the topology of $(X_g)_g$.

We have 
$$
W(\sG, \mu, \d)=\bigcup_E \bigcap_{\beta\in B\stm E} G_\beta,
$$
where $E$ runs over all subsets of $B$ such that $\mu(E)<\d$. 
Pick such a set $E\sbs B$. Put $C= \{\a: q_\a(E\cap B_\a)\ge f_\a\}$.
By the choice of $\d$, we have $p(C)<\e$, and
$$
\bigcap_{\beta\in B\stm E} G_\beta= \bigcap_{\a\in A}\bigcap_{\beta\in B_\a\stm E} G_\beta
\sbs \bigcap_{\a\in A\stm C}\bigcap_{\beta\in B_\a\stm E} G_\beta.
$$
If $\a\in A\stm C$, then $q_\a(E\cap B_\a)< f_\a$, hence 
$\bigcap_{\beta\in B_\a\stm E} G_\beta\sbs W(\sG_\a,q_\a,f_\a)=U_\a$, and we can continue the previous
displayed formula:
$$
\sbs \bigcap_{\a\in A\stm C} U_\a \sbs W(\sS,p,\e)
$$. 
This proves the inclusion $W(\sG, \mu, \d)\sbs W(\sS,p,\e)$.
\end{proof}

\section
{Refinements of locally compact group topologies}
In this Section we prove that a locally compact group $G$ is $g$-sequential
if and only if its local weight is not a $g$-sequential cardinal (Theorem
~3.14). We also prove that the product of a family of non-trivial
$g$-sequential groups is $g$-sequential if and only if the cardinality of
the family is not $g$-sequential (Theorem~3.2 and Proposition~3.3).

If a compact group
$G$ admits a strictly finer countably compact group topology, then the
weight of $G$ is an Ulam-measurable cardinal \cite{A2}. The converse of
this result
of A.~V.~\Arh\ was proved by W.~W.~Comfort and D.~Remus \cite{CR2} under the
assumption that $G$ is either Abelian or connected. We prove that this
assumption is superfluous (Theorem~3.16).
\begin{proposition}
\label{p:3.1}
A topological group $G$ is $g$-sequential if and
only if every sequentially continuous seminorm on $G$ is continuous.
\end{proposition}

\begin{proof} The topology of every topological group is defined by the set
of all continuous seminorms. If a topological group $(G, \sT)$ admits a
strictly finer group topology $\sT'$, there is a $\sT'$-continuous seminorm
$p$ on $G$ which is not $\sT$-continuous. If convergent sequences in $G$ are
the same for $\sT$ and $\sT'$, then $p$ is sequentially continuous.
Conversely, suppose $p$ is sequentially continuous but not continuous. Then
the same is true for every seminorm $p_g$ defined by $p_g(x)=p(gxg^{-1})$,
$g,x\in G$. The family of all seminorms of the form $p_g$ is invariant under
inner automorphisms and hence defines a group topology $\sT'$ on $G$. The
identity map $(G, \sT)\rar (G,\sT')$ is sequentially continuous but not
continuous, thus $(G, \sT)$ is not $g$-sequential.\end{proof}

If $\tau$ is a $g$-sequential cardinal, the compact group $2^\tau$ is not
$g$-sequential. This follows from Propositions~2.3 and~3.1. Actually
for such a cardinal
$\tau$ the space $2^\tau$ is not even $g$-stable. To see this, one
can either repeat the part (a) of the proof of Proposition~2.12
or use the result itself:
since the group $\Sym(E)$ is homeomorphic to a subspace of $2^\tau$ and
$\Sym(E)$ is not $g$-stable, it follows that $2^\tau$ is not $g$-stable
neither. We now show
that if $\tau$ is not $g$-sequential, then the group $2^\tau$ is
$g$-sequential. This is a special case of a more general theorem:
\begin{theorem} 
\label{t:3.2}
Let $\{G_\a:\a\in A\}$ be a family of $g$-sequential
topological groups. If the cardinal $\Card(A)$ is not $g$-sequential, then
the product $G=\prod_{\a\in A} G_\a$ is a $g$-sequential topological group.
\end{theorem}

\begin{proof} In virtue of Proposition~3.1 it suffices to show that every
sequentially continuous seminorm $p$ on $G$ is continuous. Clearly we may
assume that $p$ is bounded by $1$. For every $B\sbs A$ let $G_B$ be the group
$\prod_{\a\in B} G_\a$. We identify $G_B$ with the subgroup of $G$
consisting of all $\{g_\a\}$ such that $g_a$ is the unity of $G_\a$ for
each $\a\in A\setminus B$. Let $\sN(G)$ be the set of all neighbourhoods
of unity in $G$. Define a function $q:P(A)\rar \R$ by
$$
q(B)= \limsup_{x\to e, x\in G_B}p(x)=
\inf \{\sup \{p(x): x\in U\}: U\in \sN(G_B)\}.
$$
In other words, $q(B)$ is the oscillation of the restriction of $p$
to $G_B$ at the unity. Let us check that $q$ is a subadditive measure on
$A$. Suppose $B$ and $C$ are disjoint subsets of $A$. For a given $\e>0$
pick $U\in \sN(G_B)$ and $V\in \sN(G_C)$ so that $p(x)<q(B)+\e$ for every
$x\in U$ and $p(y)<q(C)+\e$ for every $y\in V$. Let $W=UV$ (the product in
the group $G$). Then $W\in\sN(G_{B\cup C)}$. If $z\in W$, then $z=xy$ for
some $x\in U$ and $y\in V$, so we have $p(z)\le p(x)+p(y)<q(B)+q(C)+2\e$.
It follows that $q(B\cup C)\le q(B)+q(C)+2\e$ and hence
$q(B\cup C)\le q(B)+q(C)$, since $\e$ was arbitrary. Thus the condition~
1 of Definition~1.3 is verified.
To verify the second condition,
suppose that $A_1\supset A_2\supset\dots$ is a decreasing sequence of
subsets of $A$ with empty intersection. We must show that $\lim q(A_n)=0$.
Let $a_n=\sup\{p(x):x\in G_{A_n}\}$. Any sequence $\{x_n\}$ such that $x_n
\in G_{A_n}$ converges to the unity. Since $p$ is
sequentially continuous, we have $\lim p(x_n)=0$. It follows that $\lim a_n
=0$. Since $q(A_n)\le a_n$, we have $\lim q(A_n)=0$.

Since each $G_\a$ is $g$-sequential, the restriction of $p$ to each $G_\a$
is continuous (Proposition~3.1). It follows that $q(B)=0$ if $B\sbs A$ is a
singleton. Since $\Card(A)$ is not $g$-sequential, every subadditive measure
on $A$ which is zero on singletons is zero on $A$. Hence $q(A)=0$. This
means that $p$ is continuous at the unity and hence everywhere, as required.
\end{proof}

The condition that $\Card(A)$ be not $g$-sequential is also necessary for
the product $G=\prod_{\a\in A} G_\a$ to be $g$-sequential, provided that
all groups $G_\a$ are non-trivial (we call a group {\it trivial\/}
if it consists of one point):
\begin{proposition}
\label{p:3.3}  
Let $\{G_\a:\a\in A\}$ be
a family of non-trivial
topological groups. If the cardinal $\Card(A)$ is $g$-sequential, then
the product $G=\prod_{\a\in A} G_\a$ is not a $g$-sequential topological
group.
\end{proposition}

\begin{proof} Let $\tau= \Card(A)$. We have observed that the space $2^\tau$
is not $g$-stable. The space $G$ contains a topological copy of $2^\tau$ and
hence is not $g$-stable either (Proposition~2.2(b)). The remark after
Definition~2.7 shows that the group $G$ is not $g$-sequential.
\end{proof}

Our next aim is Theorem~3.14, which states that
a locally compact group $G$ is $g$-sequential if and only
if its local weight is not a $g$-sequential cardinal. The
{\it weight\/} $w(X)$ of a space $X$ is defined by $w(X)=\min\{\card{\sB}:
\sB \text{ is a base for $X$}\}$, and the {\it local weight\/} of a
topological group $G$ is the cardinal $\min\{w(U): U\in \sN(G)\}$,
where $\sN(G)$ is the set of neighborhoods of the neutral element. The local
weight of a locally compact group $G$ equals its character
$\chi(G)=\min\{\card{\sB}: \sB \text{ is a base at the unity}\} $.

We shall need a few known facts concerning topology of locally compact groups.
For the reader's convenience, we sketch the proofs.

\begin{proposition}               
\label{3.4}If $G$ is a locally compact group and $\tau=\chi(G)$,
then $G$ contains a topological copy of the cube $2^\tau$.
\end{proposition}  

\begin{proof} First assume that $G$ is compact. Then $G$ is dyadic (= an
image of a Cantor cube; see \cite{U3} for an easy proof of this theorem,
due to Ivanovski\u\i\ and Kuzminov, and for its generalizations). The
$\pi$-character of a group equals its character \cite{A1}, so the
$\pi$-character of $G$ equals $\tau$. Shapirovski\u\i's theorem \cite{Shap}
implies that $G$ can be mapped onto the Tikhonov cube $I^\tau$. By a theorem
of Efimov \cite{Ef}, a dyadic compact space $X$ contains a copy of $2^\tau$
if and only if $X$ can be mapped onto $I^\tau$ (or, equivalently, if and
only if $X$ is not the union of countably many closed subspaces of weight
$<\tau$), and the proposition follows. See \cite{Shak} for another proof.

The general case reduces to the case of a compact group.
Indeed, let $K$ be a compact $G_\delta$-subgroup of $G$. Then $\chi(K)
=\tau$, so $K$ contains a copy of $2^\tau$.
\end{proof}
\begin{proposition}
\label{p:3.5} Let $G$ be a locally compact topological group.
If the local weight of $G$ is a $g$-sequential cardinal, the group $G$ is
not $g$-sequential.
\end{proposition}

\begin{proof} Let $\tau=\chi(G)$. The group $G$ contains a topological copy
of the cube $2^\tau$ (Proposition~3.4), so we can apply the same argument as in the
proof of Proposition~3.3.
\end{proof}
\begin{definition} 
\label{d:3.6}
We say that a map $f:X\rar Y$ between topological spaces $X$
and $Y$ has the {\it lifting property for convergent sequences\/}, or simply
{\it lifting property\/}, if for any point $x\in X$
and any convergent sequence
$y_1, y_2,\dots \in Y$ with the limit $y=f(x)$
there exists a sequence $x_1, x_2, \dots$ converging to $x$ and
such that $f(x_n)=y_n$ for every $n=1, 2, \dots$.
\end{definition}

It is easy to see that an open map defined on a metrizable space has the
lifting property. More generally, every open map with a metrizable
kernel has this property. We say that a map $f:X\rar Y$ has {\it metrizable
kernel\/} if there exists a map $g:X\rar M$ to a metrizable space $M$ such
that the map $(f,g): X\rar Y\times M$ is a homeomorphic embedding.
\begin{proposition}                
\label{3.7}
Let $G$ be a locally compact group, $H$ its closed
subgroup. Then the quotient map $q:G\rar G/H$ has the lifting property
for convergent sequences.
\end{proposition}

\begin{proof} 
 An open map $f:X\rar Y$ has the lifting property if $X$ can
be covered by open sets $U$ such that the restriction $f|U:U\rar f(U)$ has
this property. It follows that without loss of generality we may assume that
$G$ is $\sigma$-compact (= the union of countably many compact sets). We
represent the map $q:G\rar G/H$ as a limit
projection of an inverse system.
Let $\sK=\{K_\a: \a<\tau\}$ be the family of all
compact $G_\delta$ subgroups of $G$,
where $\tau=w(G)$. For every $\a<\tau$
let $N_\a=H\cap\bigcap_{\beta<\a}K_{\beta}$ and $X_\a=G/N_\a$. For
$\a<\beta$ a natural quotient map $p_\a^\beta: X_\beta\rar X_\a$ is defined,
so we get an inverse system $S=\{X_\a, p_\a^\beta\}$ with open bonding maps.
Every neighbourhood
of unity in $G$ contains a subgroup $K\in \sK$ \cite[theorem 8.7]{HR}, so
the inverse limit of $S$ can be identified with $G$, and $q$
can be identified with the
projection $\lim S\rar X_0$. The system $S$ is
continuous, in the sense that for every limit $\a$ the space $X_\a$ is the
inverse limit of its predecessors. Every map $p_\a^{\a+1}$ has metrizable
kernel. Indeed, $G/K_\a$ is metrizable, and the natural map
$$
X_{\a+1}=G/(N_\a\cap K_\a) \rar X_\a \times (G/K_\a)
$$
is a homeomorphic embedding. Hence every map $p_\a^{\a+1}$ has the lifting
property for convergent sequences. Given a convergent sequence in $G/H=X_0$,
we can lift it in $X_\a$ by recursion, using the continuity of $S$ at
limit steps. In the end we get a required lifting in $G$.
\end{proof}

Actually this proof shows more: the quotient map $q:G\rar G/H$ is 0-soft.
A map $p:X\rar Y$ is called {\it $n$-soft\/} \cite{Shch}, where $n$ is an
integer, if for any paracompact space $Z$ with $\dim Z\le n$, any closed
subspace $A\sbs Z$ and any commutative diagram of the form
$$
\CD
A @>f>> X\\
@ViVV @VVpV\\
Z @>>g> Y,
\endCD
$$
where $i:A\rar Z$ is the inclusion, there exists a map $h:Z\rar X$ which is
compatible with the diagram in the sense that $hi=f$ and $ph=g$. Clearly
0-softness implies the lifting property for convergent sequences: just take
for $Z$ the one-point compactification of a countable discrete space.
It follows from Michael's selection theorem for 0-dimensional spaces
\cite{Mi1} that any
open map with a completely metrizable kernel is 0-soft. (A map $f:X\rar Y$
has {\it completely metrizable kernel\/} if there exists a complete metric
space $M$ and a map $g:X\rar M$ such that the map $(f,g):X\rar Y\times M$
is a homeomorphic embedding.) The maps $p_\a^{\a+1}$ from the proof of 
Proposition~\ref{3.7} have completely metrizable kernel,
since metrizable locally compact spaces
are completely metrizable. The 0-softness of quotient maps $G\rar G/H$ for
(locally) compact groups immediately implies the Ivanovski\u\i\ --- Kuzminov
theorem mentioned above, which says that compact groups are dyadic (see
\cite{U3}). Indeed, if $X$ is a compact space such that the constant map of
$X$ to a singleton is 0-soft (such spaces are called
{\it Dugundji-compact\/} \cite{P}, \cite{Hay}), then $X$ is dyadic: there
exists an onto map $f:A\rar X$, where $A$ is a closed subspace of a Cantor
cube $Z=2^\tau$, and the assumption on $X$ implies that $f$ has an extension
$h:Z\rar X$.
\begin{proposition}  
\label{3.8}Let $H$ be a closed normal subgroup of a
topological group $G$. Suppose that the quotient map $q:G\rar G/H$ has the
lifting property for convergent sequences (Definition~3.6).
If $H$ and $G/H$ are $g$-sequential, then $G$ is also $g$-sequential.
\end{proposition}

\begin{proof} 
 Let $\sT$ be the original topology of $G$, and let $\sT'$ be
a finer group topology with the same convergent sequences.
We must prove that $\sT=\sT'$. Since $H$ is
$g$-sequential, the topologies $\sT$ and $\sT'$ agree on $H$. Since $q:
G\rar G/H$ has the lifting property, the quotient topologies $\sT/H$ and
$\sT'/H$ on $G/H$ have the same convergent sequences. Since the group $(G/H,
\sT/H)$ is $g$-sequential, the topologies $\sT/H$ and $\sT'/H$ are equal.
Now Lemma~1 of \cite{DS} implies that $\sT=\sT'$.
\end{proof}

Combining Proposition~3.8 with Proposition~\ref{3.7}, we get the following corollary:

\begin{corollary} 
\label{3.9}
If a locally compact group $G$ contains a closed
normal subgroup $H$ such that both $H$ and $G/H$ are $g$-sequential, then
$G$ is $g$-sequential.\qed
\end{corollary} 

This corollary permits to reduce the proof of Theorem~3.14 to the case of a
compact group and to consider separately the cases of connected groups and
of zero-dimensional groups.
\begin{proposition}  
\label{3.10} Let $G$ be a compact group of weight $\tau$. If
$G$ is either Abelian or connected, there exists a continuous surjective
homomorphism $\prod_{i\in I} K_i \rar G$, where $\{K_i:i\in I\}$
is a family of compact metric groups such that $|I|=\tau$.
\end{proposition} 

\begin{proof} Consider first the case when $G$ is Abelian. The groups $K_i$
that we are looking for will also be Abelian. By the Pontryagin duality,
the assertion is equivalent to the following: every Abelian group $A$ is
a subgroup of an Abelian group $B$ such that $|B|=|A|$ and $B$ is the direct
sum of a family of countable groups. This is clear: embed $A$ into a
divisible group $B$ of the same cardinality
 \cite[theorem 24.1 and proposition 26.2]{F} and note that every divisible Abelian group is the
direct sum of a family of countable groups \cite[theorem 23.1]{F}.
Now consider the case when $G$ is connected. Then $G$ is a homomorphic image
of the product of its center $Z$ and its commutator group $D$ \cite[Theorem 9.24]{HofM}. We
have just proved that $Z$ is an image of the product of a family of compact
metric groups, and the same is true for $D$, since $D$ is an image of the
product of a family of semisimple compact Lie groups \cite[Theorem 9.19]{HofM}.
\end{proof}

The case of zero-dimensional compact groups is more complicated. It seems
that the easiest way to cope with it is to invoke the notion of the free
zero-dimensional compact group. Let $X$ be a zero-dimensional compact space.
Then we can define the {\it free zero-dimensional compact group\/} $C_0(X)$
on $X$ which is characterized by the following properties: $C_0(X)$ is a
zero-dimensional compact group, $X$ is a subspace of $C_0(X)$, and every
continuous map $f:X\rar G$, where $G$ is a zero-dimensional compact group,
extends uniquely to a continuous homomorphism $\bar f:C_0(X)\rar G$. It
suffices if the last property holds for all finite groups $G$, since every
zero-dimensional compact group is a subgroup of a product of finite groups.
Hence the group $C_0(X)$ can be constructed as follows. Let $\sF$ be the
family of all continuous maps $f:X\rar G_f$ of $X$ to finite discrete
groups (considered up to an isomorphism). Let $K=\prod \{G_f: f\in \sF\}$.
There is a natural embedding $j:X\rar K$ defined by $j(x)=\{f(x):f\in\sF\}$.
The group $C_0(X)$ can be identified with the
subgroup of $K$ generated by $X$. Since $w(K)=\Card(\sF)=w(X)$, it follows
that the weight of the group $C_0(X)$ equals the weight of $X$. In
particular, the group $C_0(X)$ is metrizable if $X$ is metrizable.
\begin{proposition} 
\label{3.11}
Let $G$ be a zero-dimensional compact group of
weight $\tau$. Let $A$ be a set of cardinality $\tau$. One can assign to
every subset $B\sbs A$ a closed normal subgroup $G_B$ so that
\begin{enumerate}
\item $G_{B\cup C}=G_BG_C$ for all $B,C\sbs A$;
\item if $B_1\supset B_2\supset\dots$ is a decreasing sequence of subsets
of $A$ such that $\bigcap B_n=\emptyset$, then the sequence of subgroups
$G_{B_1}, G_{B_2}, \dots$ converges to the unity (in the sense that every
$U\in \sN(G)$ contains all but finitely many of them);
\item if $B\sbs A$ and $A\setminus B$ is countable, there exists a closed
metrizable subgroup $M$ of $G$ such that $G=G_BM$.
\end{enumerate}
\end{proposition}

\begin{proof} The property of $G$ that we must prove is preserved by
homomorphisms, so we may replace $G$ by any group which admits a
homomorphism onto $G$. Let $X=2^A$. Then $G$ is homeomorphic to $X$
\cite[theorem 9.15]{HR}
(for our purposes it would suffice to know that $G$ is dyadic),
so there exists a continuous homomorphism of the free zero-dimensional group
$C_0(X)$ onto $G$. Thus we may confine ourselves to the case $G=C_0(X)$.
For every $B\sbs A$ let $X_B=2^B$, and let $p_B:X\rar X_B$ be the natural
projection. The map $p_B$ extends to a homomorphism $C_0(p_B):G\rar
C_0(X_B)$. Let $K_B$ be the kernel of this homomorphism, and let
$G_B=K_{A\setminus B}$. We show that the assignment $B\mapsto G_B$ is as
required.

Given a subset $B\sbs A$, let us say that a map $f$ defined on $X$ is
{\it $B$-nice} if for any $x, y\in X$ with $p_B(x)=p_B(y)$ we have
$f(x)=f(y)$. Our proof leans on the following Assertion:

\smallskip
\it Let $q:G\rar H$ be a continuous homomorphism, and let $B$ be a subset of
$A$. The kernel of $q$ contains the subgroup $K_B$ if and
only if the restriction of $q$ to $X$ is $B$-nice.
\smallskip\rm

This is clear, since the kernel of $q$ contains $K_B$ if and only if $q$
can be written as the composition $q= q'C_0(p_B)$ for some homomorphism
$q':C_0(X_B)\rar H$, and a map $f:X\rar H$ is $B$-nice if and only if
$f$ can be written as the composition $f=f'p_B$ for some $f':X_B\rar H$.

We must check that
\begin{enumerate}
\item $K_{B\cap C}=K_BK_C$ for all $B,C\sbs A$;
\item if $B_1\sbs B_2\subset\dots$ is an increasing sequence of subsets
of $A$ such that $\bigcup B_n=A$, then the sequence of subgroups
$K_{B_1}, K_{B_2}, \dots$ converges to the unity;
\item if $B\sbs A$ and $B$ is countable, there exists a closed
metrizable subgroup $M$ of $G$ such that $G=K_BM$.
\end{enumerate}

We check 1. It suffices to prove that for any homomorphism
$q:G\rar H$ the kernel of $q$ contains $K_{B\cap C}$ if and only if it
contains $K_B$ and $K_C$. In virtue of the Assertion, this is equivalent
to the following: a map $f:X\rar H$ is $(B\cap C)$-nice if and only if it
is $B$-nice and $C$-nice. The last assertion is obviously true: if, say,
$f$ is $B$-nice and $C$-nice, then for any $x, y\in X$ such that
$p_{B\cap C}(x)=p_{B\cap C}(y)$ we can find $z\in X$
so that $p_B(z)=p_B(x)$ and $p_C(z)=p_C(y)$, and then $f(x)=f(z)=f(y)$,
which means that $f$ is $(B\cap C)$-nice.

We check 2. Since $G$ is a zero-dimensional compact group, the kernels
of homomorphisms of $G$ to finite groups form a base at the unity. Hence it
suffices to prove the following: if $q:G\rar H$ is a homomorphism to a
finite group $H$ and $B_1\sbs B_2\subset\dots$ is an increasing sequence of
subsets of $A$ such that $\bigcup B_n=A$, then the subgroup $K_{B_n}$ is
contained in the kernel of $q$ for some $n$ (and hence also for all larger
indices). Equivalently (see the Assertion), we must prove that every
continouos
map $f:X\rar H$ is $B_n$-nice for some $n$. Each fiber
of $f$ is a closed-and-open subset of $X$, hence the union of finitely many
basic closed-and-open sets. It follows that $f$ is $C$-nice for some finite
set $C\sbs A$. Since $\bigcup B_n=A$, we have $C\sbs B_n$ if $n$ is
sufficiently large, and then $f$ is $B_n$-nice, as required.

We check 3. Let $B$ be a countable subset of $A$. The map
$p_B:X\rar X_B$ has a right inverse $s:X_B\rar X$, hence the homomorphism
$C_0(p_B):G\rar C_0(X_B)$ also has a right inverse $C_0(s):C_0(X_B)\rar G$.
It follows that $G$ is the semidirect product of the kernel $K_B$ of $p_B$
and the range of $C_0(s)$, say $M$, which is isomorphic to
$C_0(X_B)$. In the paragraph preceding the proposition that we are proving
we noted that the functor $C_0$ preserves metrizability. It follows that
the group $C_0(X_B)$ is metrizable, and so is the group $M$.
\end{proof}

Proposition~3.11 allows to apply the argument used in the proof of
Theorem~3.2 to zero-dimensional compact groups. We need
two more propositions for the proof of Theorem~3.14.
\begin{proposition} 
\label{3.12}
Let $p$ be a subadditive measure on a set $A$
(Definition~1.3). If the cardinal $\Card(A)$ is not $g$-sequential, there
exists a countable subset $B\sbs A$ such that $p(A\setminus B)=0$.
\end{proposition} 

\begin{proof} Since the group $2^A$ is $g$-sequential (Theorem~3.2),
it follows from Propositions~2.3 and~3.1
that $p$ is continuous on $2^A$. Hence the set
$p^{-1}(0)$ is of the type $G_\delta$. This implies what we need.
\end{proof}
\begin{proposition} 
\label{3.13}
Let $G=HK$ be a compact group,
where $H$ and $K$
are closed subgroups. If $U\in \sN(H)$ and $V\in \sN(K)$, then
$UV\in\sN(G)$.
\end{proposition} 

\begin{proof} Let $P=H\times K$. We must show that the map $f:P\rar G$
defined by $f(x,y)=xy$ is open. Let $L=H\cap K$. Consider the right action of
$L$ on $P$, defined by $(x,y).g=(xg, g^{-1}y)$. The map $f$ can be
identified with the canonical map $P\rar P/L$ of $P$ onto the orbit space
$P/L$. 
Since the latter map is open, so is $f$.
\end{proof}

We are now ready to prove Theorem~3.14, which shows that whether a locally
compact group is $g$-sequential or not depends only on its local weight.
\begin{theorem} 
\label{3.14}
A locally compact group $G$ is $g$-sequential if
and only if the local weight of $G$ is not a $g$-sequential cardinal.
\end{theorem} 

\begin{proof} The implication in one direction was established in Proposition
~3.5. Now suppose that the local weight of $G$ is not a $g$-sequential
cardinal. We must show that the group $G$ is $g$-sequential. Clearly a group
is $g$-sequential if it contains an open $g$-sequential subgroup, so we may
assume that $G$ is $\sigma$-compact. In this case $G$ contains a compact
normal subgroup $K$ such that the quotient group $G/K$ is metrizable (and
hence $g$-sequential) \cite[Theorem 8.7]{HR}. Corollary~3.9
implies that we may confine ourselves to the case when $G$ is compact.
Let $C$ be the connected component of the unity in $G$. The quotient group
$G/C$ is zero-dimensional, and another application of Corollary~3.9 shows
that we may assume that $G$ is either connected or zero-dimensional. In the
first case there exists a surjective homomorphism $P\rar G$, where $P$ is
the product of a family $\ga$ of compact metric groups such that $\Card(\ga)
=w(G)$ (Proposition~3.10). In virtue of Theorem~3.2, $P$ is $g$-sequential,
hence so is $G$. Now consider the case when $G$ is zero-dimensional. Let
$A$ be a set of cardinality $w(G)$, and let $B\mapsto G_B$ be an
assignment with the properties described in Proposition~3.11. In order
to prove that every sequentially continuous seminorm $p$ on $G$ is
continuous (Proposition~3.1), we can apply the
same construction as in the proof of Theorem~3.2. Just as in that proof,
define a function $q:P(A)\rar \R$ by
$$
q(B)=\inf \{\sup \{p(x): x\in U\}: U\in \sN(G_B)\}.
$$
Then $q$ is a subadditive measure on $A$. To prove that $q(B\cup C)\le
q(B)+q(C)$, use the property 1 of Proposition~3.11 and Proposition~3.13.
To prove that $\lim q(B_n)=0$ whenever $B_1\supset B_2\supset\dots$ and
$\bigcap B_n=\emptyset$, use the property 2 of Proposition~3.11. Since
the cardinal $\Card(A)=w(G)$ is not $g$-sequential, Proposition~3.12 implies
that $q(B)=0$ for some subset $B\sbs A$ such that $A\setminus B$ is
countable. Pick a closed
metrizable subgroup $M\sbs G$ so that $G=G_BM$ (property 3 of
Proposition~3.11). The restriction of $p$ to $G_B$ is continuous since $q(B)
=0$, and the restriction of $p$ to $M$ is continuous since $M$ is
metrizable. The argument used to prove the inequality $q(B\cup C)\le
q(B)+q(C)$, based upon Proposition~3.13, shows that $p$ is continuous
on $G$.
\end{proof}

Theorem~3.14 answers the question of what locally compact groups admit
a strictly finer group topology with the same convergent sequences. We now
show that a locally compact group admits a strictly finer group topology
which agrees with the original one on every countable set if and only if
the local weight of $G$ is Ulam-measurable. As observed by A.~V.~\Arh\ \cite{A2},
in one direction this follows from the results of \cite{V}. If $(G,\sT)$ is
a locally compact group and $\sT'$ is a strictly finer group topology which
agrees with $\sT$ on countable sets, then the group $(G,\sT')$ is locally
countably compact, hence its completion $H$ is locally compact. Since the
natural homomorphism of $G$ to $H$ is sequentially continuous but not
continuous, the local weight of $(G,\sT)$ must be Ulam-measurable \cite{V}. We
now establish the converse (Theorem~3.16).

The definition of the $g$-modification $\sT_g$ of a topology $\sT$ on a set
$X$
(Definition~2.1) depended on the notion of a subadditive measure. We get the
definition of the {\it $m$-modification} $\sT_m$ of a topology $\sT$ if we
replace subadditive measures by two-valued measures. In other words, the
topology $\sT_m$ can be described as follows.
Let $\sS=\{U_\a: \a \in A\}$ be a family of $\sT$-open sets in $X$, and let
$p$ be a $\o_1$-complete ultrafilter on the index set $A$. Let
$$
W(\sS,p)=\bigcup_{B\in p} \bigcap_{\a\in B} U_\a.
$$
be the set of all $x\in X$ such that the set
$\{\a\in A: x\in U_\a\}$ is in $p$. The sets $W(\sS,p)$ form a base for
the topology $\sT_m$.

All properties of the $g$-modification $\sT_g$ of a topology $\sT$
established in Propositions~2.2 and~2.4--2.6 hold also for the
$m$-modification $\sT_m$. Note that $\sT\sbs \sT_m\sbs \sT_g$.
\begin{proposition} 
\label{p:3.15}Let $(X, \sT)$ be a topological space. The
topologies $\sT$ and $\sT_m$ agree on every subset of $X$ of non-Ulam-measurable
cardinality.
\end{proposition}

\begin{proof} It suffices to prove that $\sT_m=\sT$ if $\Card(X)$ is
non-Ulam-measurable. Let $W(\sS,p)$ be a basic open set for $(X,\sT_m)$, where
$\sS=\{U_\a: \a \in A\}$ is a family of open sets in $(X,\sT)$ and $p$ is a
$\omega_1$-complete ultrafilter on $A$.  Since
$\Card(X)$ is non-Ulam-measurable, the cardinal $\Card(\sT)\le
2^{\Card(X)}$ is also non-Ulam-measurable. It follows that there is $B\in p$ and
$U\in \sT$ such that $U_\a=U$ for every $\a\in B$. Then $W(\sS,p)=U$. Thus
$\sT_m=\sT$.
\end{proof}

Let us say that a space $(X,\sT)$ is {\it $m$-stable} if $\sT=\sT_m$.
Proposition~3.15 means that every space of non-Ulam-measurable cardinality is
$m$-stable. On the other hand, it is easy to see that the cube $2^\tau$ is
not $m$-stable if $\tau$ is a measurable cardinal.
\begin{theorem} 
\label{3.16}
Let $G$ be a locally compact group of Ulam-measurable
local weight. Then
\begin{enumerate}
\item $G$ admits a strictly finer group topology which agrees with the
original one on every set of non-Ulam-measurable cardinality;
\item there exists a locally compact group $H$ and a sequentially continuous
homomorphism $G\rar H$ which is not continuous.
\end{enumerate}
\end{theorem}

\begin{proof} The group $G$ contains a topological copy of the cube $2^\tau$,
where $\tau=\chi(G)$ (Proposition~3.4), hence $G$ is not $m$-stable.
It follows that
the $m$-modification $\sT_m$ of the original topology $\sT$ of $G$ is as
required. We noted above that the second part of the theorem follows from
the first: we can take for $H$ the completion of $(G, \sT_m)$.
\end{proof}

\begin{corollary} 
\label{3.17}Every compact group of Ulam-measurable weight admits
a strictly finer countably compact group topology.\qed
\end{corollary}

This answers Question~5.4(a) from \cite{CR2}.

\section
{On subgroups of $g$-sequential groups}

We now return to \Arh's problem: is the variety $\fV$ generated by free
topological groups on metrizable spaces proper? If there exist
real-measurable cardinals, the answer is no (Theorem~2.13). The following
approach to a possible positive solution was suggested in \cite{MNPS}.

To prove that the variety $\fV$ coincides with the class of all topological
groups, it suffices to show that $\fV$ contains the free topological group
$F(X)$ for every Tikhonov space $X$. Embed $X$ in a product $P$ of metric
spaces so that every continuous pseudometric on $X$ extends to a continuous
pseudometric on $P$. Assume that the following assertions are true:
\begin{itemize}
\item[(H)] Any product of metric spaces is $f$-sequential (Definition~1.4).
\item[(F)] If $Y\sbs Z$ and every continuous pseudometric on $Y$ extends
to a continuous pseudometric on $Z$, then the natural map $F(Y)\to F(Z)$ is a
topological embedding.
\end{itemize}
The assertion (F) was formulated in the Introduction in an equivalent form
involving fine uniformities.
It is easy to see that the free topological group $F(Y)$ is $g$-sequential
if and only if the space $Y$ is $f$-sequential, hence (H) implies that
the group $F(P)$ is $g$-sequential. Now (F) implies that $F(X)$ is a
subgroup of a $g$-sequential group and hence $F(X)\in \fV$.

In this argument we tacitly used the equivalence of the conditions
1--3 of Definition~1.2, which is Theorem~3.7 in \cite{MNPS}. For the
reader's convenience we reproduce the proof. If $X$ is metrizable (or just
$f$-sequential), then every sequentially continuous homomorphism $F(X)\rar
G$ is continuous, since its restriction to $X$ is continuous. Thus free
groups on metrizable spaces have the property 1 of Definition~1.2,
and so do their quotient groups, since the property 1 is preserved by
quotients. This means that 3~ $\Rightarrow$~1. The implication
1~$\Rightarrow$~2 is clear. To prove that 2~$\Rightarrow$~3,
note that every topological space $X$ can be represented as the image
of a metric space $M$ under a continuous map $f:M\rar X$ such that every
convergent sequence in $X$ is the image of a convergent sequence in $M$:
just take for $M$ the disjoint sum of sufficiently many convergent
sequences. If $G$ is a topological group and $f:M\rar G$ is as above, then
the homomorphism $\bar f:F(M)\rar G$ which extends $f$ is quotient if and
only if $G$ does not admit a strictly finer group topology with the same
convergent sequences. Hence 2~$\Rightarrow$~3.

Now let us discuss the assertions~(H) and~(F). If there are sequential
cardinals, then (H) cannot be true: the product of a family $\ga$ of
non-trivial metric spaces is not $f$-sequential if $\Card(\ga)$ is
sequential. We prove that if there are no sequential cardinals, then (H)
is true. More generally,
the product of a family $\ga$ of metric spaces is
$f$-sequential if $\Card(\ga)$ is not sequential (Theorem~4.4). For products
of {\it separable\/} metric spaces this was proved in \cite[theorem 1.5]{ACh}. 
In the non-separable case we apply the technique of \cite{U4}.

As far as (F) is concerned, this is the main result of \cite{Sip, S2}.
Whether or not one shares
the opinion of V.\,Pestov that
``Sipacheva's proofs gradually became commonly recognized as correct"
\cite[p.~108]{CHR}, it would
be desirable to find an independent proof of the fact that under some
additional set-theoretic assumptions every free topological group $F(X)$ is
a subgroup of a $g$-sequential group. To prove that the variety $\fV$
coincides with the class of all topological groups, it would actually
suffice to prove the last assertion only for some very special spaces $X$.
Call a topological space X {\it ultrasimple\/} if it is either (1) discrete
or (2)
contains exactly one non-isolated point $p$, and the trace of the filter of
neighbourhoods of $p$ on $X\setminus\{p\}$ is an ultrafilter. Call a space
{\it simple\/} if it is a disjoint sum of ultrasimple spaces. It is easy to
see that every topological space is an image of a simple space under a
quotient map. It follows that a variety $\fW$ of topological groups contains
all topological groups if and only if $F(X)\in \fW$ for every simple space
$X$. Thus to prove that the variety $\fV$ contains all topological groups
it would suffice to show that the group $F(X)$ is a subgroup of a
$g$-sequential group for every simple space $X$. This observation motivates
the following question:

\begin{question}
\label{4.1}
Let $X$ be a simple space. Is the group $F(X)$ isomorphic to a topological
subgroup of $F(M)$ for some metric space $M$?
\end{question}

Theorem~2.13 implies that the answer is no if there are large cardinals, but
it is not clear if the answer can be consistently yes. The question seems to
be closely related to the problem considered in \cite{DW}.

The question remains open whether every topological group is a subgroup of
a $g$-sequential group (assuming there are no $g$-sequential cardinals). We
obtain some partial results in this direction.
We prove that the group $\Aut I^\tau$ of all self-homeomorphisms of a
Tikhonov cube $I^\tau$ is $g$-sequential if and only if the
cardinal $\tau$ is not $g$-sequential (Theorem~4.11). Similarly, the unitary
group $U(H)$ of a Hilbert space $H$ is $g$-sequential if and only if the
weight of $H$ is not $g$-sequential. These results do not answer the
above-mentioned question, since not every topological group can be embedded
in a unitary group $U(H)$, and it is not known whether every topological
group can be embedded in a group of the form $\Aut I^\tau$.

We use the techniques of \cite{U4} to prove Theorems~\ref{t:4.4} and ~\ref{t:4.11}. Our
proofs
depend on the notion of the functional tightness of a topological space, due
to A.V.\Arh\ \cite{A3}. A function $f:X\rar Y$ is
{\it $\omega$-continuous\/}
if for every countable subset $A\sbs X$ the restriction $f|A:A\rar Y$ is
continuous. A space $X$ has {\it countable functional tightness\/} if every
$\omega$-continuous function $f:X\rar \R$ is continuous. Recall that a space
$X$ has {\it countable tightness\/} if for any $x\in X$ and $A\sbs X$ such
that $x\in \bar A$ (where the bar denotes the closure) there is a countable
$B\sbs A$ such that $x\in\bar B$. Clearly every space of countable tightness
has countable functional tightness. The converse is not true. For example,
a power $\R^\tau$ of the real line has countable tightness if and only if
$\tau=\omega$, and
it has countable functional tightness if and only if the cardinal $\tau$ is
non-Ulam-measurable \cite{U4}. The last assertion is a special case of the main
result of \cite{U4}: the function space $C_p(X)$ has countable functional
tightness if and only if $X$ is realcompact. Here $C_p(X)$ denotes the space
of continuous functions on $X$, endowed with the topology of pointwise
convergence. Recall that the space $C_p(X)$ has countable tightness if and
only if all finite powers of $X$ are Lindel\"of \cite[theorem 4.1.2]{A4}.
In particular, the space $C_p(X)$ has countable tightness if $X$ is compact.

The proof of the main theorem in \cite{U4} was based on a lemma which
we reproduce here for convenience. For a subset $A$ of a topological space
$X$ denote by $[A]_\omega$ the {\it $\omega$-closure\/} of $A$, that is, the
set $\bigcup\{\bar B: B\sbs A \text{ and $B$ is countable}\}$.

\begin{lemma}[\cite{U4}]
\label{l:4.2}
Let $Y$ be a space of countable functional
tightness, and let $p:Y\rar X$ be a surjective continuous map. Suppose there
is a base $\sB$ for $Y$ such that for every $U\in \sB$ there is an open set
$V$ in $X$ such that
\begin{equation}
p(U)\sbs V\sbs [p(U)]_\omega.                       \tag {$\ast$}
\end{equation}
Then $X$ has countable functional tightness.
\end{lemma}

\begin{proof} We must show that every $\omega$-continuous function $f:X\rar
\R$ is continuous. Let $x\in X$, and let $O$ be an open interval in $\R$
containing $f(x)$. Pick $y\in Y$ such
that $p(y)=x$. The function $fp:Y\rar \R$ is $\omega$-continuous and hence
continuous, so there is $U\in \sB$ such that $y\in U$ and $fp(U)\sbs O$.
Since $f$ is $\omega$-continuous, it maps the $\omega$-closure
$[p(U)]_\omega$ of the set $p(U)$ into $\bar O$. Pick an open $V$ in $X$ so
that $p(U)\sbs V\sbs [p(U)]_\omega$. Then $V$ is a neighbourhood of $x$ and
$f(V)\sbs \bar O$. Thus $f$ is continuous.
\end{proof}

This lemma leads to the following sufficient condition for a space to have
countable functional tightness.
\begin{proposition} 
\label{p:4.3}
Let $X$ be a topological space, and let
$\{d_\a:\a\in A\}$ be a family of pseudometrics on $X$ generating the
topology of $\sT$ of $X$. Assume that for any two points $x,y\in X$ and any
finite set $F$ of free ultrafilters on $A$ there exists a sequence
$y_0,y_1,\dots \in X$ such that $\lim y_n=y$ and for every $n\in \omega$
the set
$$
\{\a\in A: d_\a(y_n,x)=0\}
$$
is in the filter $\bigcap F$.
Then $X$ has countable functional tightness.
\end{proposition} 

\begin{proof} In virtue of
Lemma~\ref{l:4.2}, it suffices to construct a finer topology $\sT'$ on $X$ such that
the space $Y=(X,\sT')$ has countable tightness and the identity map
$Y\rar X$ has the property described in the lemma. We may assume that the
pseudometrics $d_\a$ are $\le 1$: otherwise replace $d_\a$ by
$\min(d_\a,1)$. Let $\beta A$ be the set of all
ultrafilters on $A$, considered as a compact space (the \v Cech --- Stone
compactification of the discrete space $A$). We identify $A$ with its image
under the natural map $A\to\beta A$. For each ultrafilter $p\in
\beta A$ define a pseudometric $d_p$ on $X$ as the $p$-limit of
pseudometrics $d_\a$. In other words, if $x,y\in X$, then
$d_p(x,y)$ is determined by the requirement that the function $p\mapsto
d_p(x,y)$ must be continuous on $\beta A$. Let $\sT'$ be the topology
on $X$ generated by all pseudometrics $d_p$, $p\in \beta A$, and let
$Y=(X,\sT')$. Let us check that the identity map $Y\rar X$ satisfies
the condition of the lemma.

For every finite set $K\sbs\beta A$ let $d_K=\max\{d_p:p\in K\}$. The
collection $\sB$ of all open $d_K$-balls, where $K$ runs over the set of
all finite subsets of $\beta A$, is a base for $Y$. Thus it suffices to
check that for every $U\in \sB$ there exists an open $V$ in $X$ such that
the condition $(*)$ of the lemma holds.

Let $U\in \sB$. Pick a finite set $K\sbs \beta A$, a point $x\in X$ and $\e>
0$ so that $U=\{z\in X: d_K(z,x)<\e\}$.
Let $L=K\cap A$, and let $V=\{z\in X: d_L(z,x)<\e\}$.
Then $V$ is open in $X$. We show that $V$ is as required: $U\sbs V \sbs
[U]_\omega$, where $[U]_\omega$ denotes the $\omega$-closure of $U$ in $X$.

The inclusion $U\sbs V$ is obvious. To prove that $V\sbs [U]_\omega$, it
suffices to show that every
$y\in V$ is the limit of a sequence $y_0, y_1\dots\in U$. Let $F=K\setminus
L$ be the set of all free ultrafilters in $K$, and let $s=\bigcap F$.
By the assumption, there is a sequence $y_0, y_1\dots\in X$ converging to
$y$ and such that for any $n\in \omega$ the set $\{\a\in A: d_\a(y_n,x)=0\}$
belongs to $s$. Then $d_p(y_n,x)=0$ for every ultrafilter $p\in F$, hence
$d_F(y_n,x)=0$ and $d_K(y_n,x)=d_L(y_n,x)$. Since $y=\lim y_n\in V$, we have
$d_K(y_n,x)=d_L(y_n,x)<\e$ if $n$ is sufficiently large. This means that
$y_n\in U$. Thus $V\sbs [U]_\omega$.

In virtue of Lemma~\ref{l:4.2}, the proof will be complete if we show that the space
$Y$ has countable functional tightness. We prove more: the tightness of $Y$
is countable. Let $x\in X$. Consider the map $f_x:X\rar C_p(\beta A)$,
defined by $f_x(y)(q)=d_q(x,y)$. It follows from the definition of the
topology $\sT'$ of the space $Y$ that $x$ is in the $\sT'$-closure of a
subset $E$
of $X$ if and only if the zero function $f_x(x)$ is in the closure of the
set $f_x(E)$ in the space $C_p(\beta A)$. Since the tightness of $C_p(\beta
A)$ is countable, it follows that the tightness of $Y$ is also countable.
\end{proof}

We now apply Proposition~\ref{p:4.3} to show that certain spaces have countable
functional tightness. Recall that a space $X$ is {\it bi-sequential\/}
\cite{Mi2} if any ultrafilter on $X$ converging to a point $x\in X$
contains a filter with a countable basis converging to the same point.
All first-countable spaces are bi-sequential.
An onto map $f:X\rar Y$ is {\it bi-quotient\/} \cite{Mi2}
if for any
$y\in Y$ and any cover $\sW$ of $f^{-1}(y)$ by open subsets of $X$
there exists a finite subfamily $\sE\sbs\sW$ such that $y$ belongs
to the interior of $f(\bigcup\sE)$. Bi-sequential spaces are precisely
the images of metric spaces under bi-quotient maps \cite{Mi3}. Bi-quotient
maps are preserved by arbitrary products \cite{Mi2} (see \cite{U2} for
a similar assertion concerning so-called tri-quotient maps).

%

\begin{theorem} 
\label{t:4.4}
Let $\{X_\a:\a\in A\}$ be a family of bi-sequential
spaces, and let $X=\prod_{\a\in A} X_\a$.
\begin{enumerate}
\item If $\Card(A)$ is non-Ulam-measurable, the product $X$ has countable
functional tightness.
\item If $\Card(A)$ is non-sequential, the product $X$ is $f$-sequential.
\end{enumerate}
\end{theorem} 

\begin{proof}
The theorem reduces to the case when all spaces
$X_\a$ are metrizable. Indeed, for every $\a\in A$ there exist a metric
space $Y_\a$ and a bi-quotient onto map $f_\a:Y_\a\rar X_\a$ \cite{Mi3}.
The product $\prod_\a f_\a:\prod_\a Y_\a\rar \prod_\a X_\a$ is bi-quotient
\cite{Mi2} and hence quotient. Since spaces of countable tightness and
$f$-sequential spaces are preserved by quotient maps, it suffices to show
that the product $\prod_\a Y_\a$ is of countable tightness or
$f$-sequential, respectively.

So assume that the spaces $X_\a$ are metrizable.
For every $\a\in A$ let $\rho_\a$ be a compatible
metric on $X_\a$. Let $d_\a$ be the pseudometric on $X$ defined by
$d_\a(\{x_\a\},\{y_\a\})=\rho_\a(x_\a,y_\a)$. The family
$D=\{d_\a:\a\in A\}$ generates the topology of $X$. We show that if
$\Card(A)$ is non-Ulam-measurable, then the family $D$ satisfies the condition
of Proposition~\ref{p:4.3}.

Let $x,y\in X$, and let $F$ be a finite set of free ultrafilters on $A$.
Since the cardinal $\Card(A)$
is assumed to be non-Ulam-measurable, free ultrafilters on $A$ are not
$\omega_1$-complete. Hence there exists a decreasing sequence
$B_0\supset B_1\supset\dots\in \bigcap F$ of subsets of
$A$ such that $\bigcap B_n=\emptyset$.
Define $y_n\in X$
by $y_{n,\a}=y_\a$ if $\a\in A\setminus B_n$ and $y_{n,\a}=x_\a$ if $\a\in
B_n$. Since $\bigcap B_n=\emptyset$, the sequence $\{y_n\}$ converges to
$y$. For every $n\in \omega$ the set $\{\a\in A:d_\a(y_{n,\a},x_\a)=0\}$
contains $B_n$ and hence belongs to $\bigcap F$. Proposition~\ref{p:4.3} implies
that $X$ has countable functional tightness.

The second part of the theorem reduces to the first. Assume that the
cardinal $\Card(A)$ is not
sequential. We must prove that every sequentially continuous function
$f:X\rar \R$ is continuous. In virtue of the first part of the theorem, it
suffices to show that $f$ is $\omega$-continuous. Since every countable
subset of $X$ lies in a product $\prod Y_\a$, where $Y_\a$ is a countable
subset of $X_a$ for every $\a\in A$, we may assume without loss of
generality that all spaces $X_\a$ are separable (or even countable), and in
this case the result is known \cite[Theorem 1.5]{ACh}. For completeness we
sketch the proof.

For each $\a\in A$ pick a point $z_\a\in X_\a$ , and let $S$ be the
$\Si$-product of the family $\{X_\a\}$ with the base-point $\{z_\a\}$,
$$
S=\{\{x_\a\}\in X:x_\a=z_\a\text{ for all but countably many }\a\in A\}.
$$
Then $S$ is a Fr\'echet space, so every sequentially continuous function
$f:X\rar \R$ is continuous on $S$, and hence the restriction of $f$ to $S$
depends on countably many
coordinates \cite{A5}. The assumption that the cardinal $\tau=\Card(A)$ is
not sequential means that any sequentially continuous function on the cube
$2^\tau$ is continuous. Since $X$ is covered by copies $F$ of this cube
such that the intersection $F\cap S$ is dense in $F$, it follows that $f$
depends on countably many coordinates and hence is continuous.
\end{proof}

Sequentially continuous functions on products were first studied by
S.\,Mazur \cite{Ma}, who proved that the first sequential cardinal (if it
exists) is weakly inaccessible (that is, regular limit).
Theorem~\ref{t:4.4} implies that if there are no sequential cardinals, then every
Tikhonov space is a subspace of an $f$-sequential Tikhonov space.
Conversely, if there exists a sequential cardinal $m$, then the cube
$I^m$ cannot be embedded in an $f$-sequential Tikhonov space $X$. Indeed,
if $I^m\sbs X$, then there is a retraction $r:X\rar I^m$. Let $f:I^m\rar
\R$ be a sequentially continuous discontinuous function. Then
the composition $fr:X\rar \R$ is sequentially continuous but not continuous,
so $X$ is not $f$-sequential.

We now establish a counterpart of Theorem~\ref{t:4.4} for topological groups
(Theorem~\ref{t:4.11}). We need some preparations.

A homomorphism $f:G\rar H$ of topological groups {\it splits\/} if there
exists a homomorphism $s:H\rar G$ which is right inverse to $f$, that is
$fs=\operatorname{id}_H$. In this case $G$ is a semidirect product of
$H$ and the kernel of $f$. An inverse system $S=\{X_\a, p_\a^\beta:\a,\beta
<\tau\}$ of topological spaces is {\it continuous\/} if for every limit
ordinal $\a<\tau$ the space $X_\a$ is naturally homeomorphic to the inverse
limit of the system $\{X_\beta:\beta<\a\}$.
\begin{definition} 
\label{4.5}
Let $\{H_\a: \a< \tau\}$ be
a well-ordered family of topological groups.
We say that a topological group $G$ is an
{\it iterated semidirect product\/} of the groups
$H_\a$, and write $G=\bigotimes_{\a<\tau} H_\a$, if there exists a
continuous inverse system
$$
\{e\}=G_0\lar\dots\lar G_\a\lar G_{\a+1}\lar\dots\lar G
$$
of topological groups $G_\a$ and homomorphisms $p_\a^\beta:G_\beta\rar G_\a$
for $\a<\beta$ such that $G=\varprojlim G_\a$ and for every $\a<\tau$ the
homomorphism $p_\a^{\a+1}$ splits, and its kernel is isomorphic to $H_\a$.
\end{definition}

In the notation of Definition~\ref{4.5}, if we pick for every $\a<\tau$ a
homomorphism $s_\a:G_\a\rar G_{\a+1}$ which is right inverse to
$p_\a^{\a+1}$, then each $H_\a$ can be identified with a subgroup of $G$,
and $G$ can be identified with the product
$\prod H_\a$ as a topological space but not as a group. If $H_\a'$ is a
closed subgroup of $H_\a$ for every $\a<\tau$, the product $\prod H_\a'$
is a subgroup of $G$ if and only if $H_\a'$ normalizes $H_\beta'$ whenever
$\a<\beta$. If $B$ is a set of ordinals $<\tau$, let $G_B=\prod H_\a'$,
where
$H_\a'=H_\a$ for $\a\in B$ and $H_\a'=\{e\}$ otherwise. Then $G_B$ is a
subgroup of $G$.

Iterated semidirect products appear naturally as groups of
self-homeo\-mor\-phisms of products of compact spaces.
For a compact space $K$ let $\Aut K$ be the group of all
self-homeomorphisms of $K$, equipped with the compact-open topology.

\begin{definition} 
\label{4.6}
Let $X$ and $Y$ be compact spaces, and let $p:X\times Y\rar X$
be the
projection. We say that a self-homeomorphism $g\in \Aut (X\times Y)$ of
the product $Z=X\times Y$ is {\it $X$-special\/} if there exists a
self-homeomorphism $h\in \Aut X$ such that $p(gz)=h(pz)$ for every $z\in Z$.
\end{definition} 

A self-homeomorphism $g\in\Aut Z$ is $X$-special if and only if it permutes
fibers of $p$.

For a topological group $G$ and a compact space $X$ let $C(X, G)$ be the
topological group
of all continuous maps from $X$ to $G$, endowed with the compact-open
topology.

\begin{lemma}
\label{l:4.7}Let $X$ and $Y$ be compact spaces. The topological
group $G$ of all $X$-special self-homeo\-mor\-phisms of $X\times Y$
is the topological semidirect
product of the groups $\Aut X$ and $C(X, \Aut Y)$. The natural homomorphism
$G\rar \Aut X$ has a natural right inverse.
\end{lemma}

\begin{proof} The kernel $K$ of the natural homomorphism $q:G\rar \Aut X$
consists
of all $g\in \Aut (X\times Y)$ which leave invariant each fiber of the
projection $X\times Y\rar X$. Every such $g$ defines a map $X\rar \Aut Y$
in a natural way. The exponential law for function spaces
\cite[Theorem 3.4.8]{En} implies
that we thus obtain a homeomorphism between $K$ and
$C(X, \Aut Y)$. On the other hand, the homomorphism $q$ has a natural
right inverse
homomorphism $s:\Aut X\rar G$, defined by $s(f)(x,y)=(f(x), y)$.
\end{proof}

Now let $\{Y_\a:\a<\lambda\}$ be a well-ordered family of compact spaces.
Let $Y=\prod_{\a<\lambda} Y_\a$, and for
every $\a\le\lambda$ let $X_\a=\prod_{\beta<\a} Y_\a$.

\begin{definition} 
\label{4.8}Let the notation be as above.
We say that a self-homeomorphism $g\in \Aut Y$ is {\it special\/} if
it is $X_\a$-special
for each $\a<\lambda$ (if $Y$ is considered as the product $X_\a\times
\prod_{\beta\ge\a} Y_\beta$).
\end{definition}

Let $G$ be the group of all special $g\in
\Aut Y$. For every $\a<\lambda$ let $G_\a$ be the group of all special
$g\in \Aut X_\a$ (this notion is well-defined, since $X_\a$ is the product of
a well-ordered family of compact spaces). There are natural homomorphisms
$G\rar G_\a$ and $G_\beta\rar G_\a$ for $\a<\beta<\lambda$. We thus obtain
an inverse system of the groups $G_\a$ with the limit $G$.
Lemma~\ref{l:4.7} shows that each $G_{\a+1}$ is the topological semidirect
product of $G_\a$ and $C(X_\a, \Aut Y_\a)$, and the homomorphism
$G_{\a+1}\rar G_\a$ has a natural right inverse. This proves the following:

\begin{lemma}
\label{l:4.9}
With the notation as above, the group $G$ of all special self-homeomorphisms
of $Y$ is an iterated semidirect product of the groups
$C(X_\a, \Aut Y_\a)$, $\a<\lambda$. Each of the groups $C(X_\a, \Aut Y_\a)$
can be identified with a subgroup of $G$ in a canonical way.
\qed\end{lemma}

Let $X$ and $Y$ be compact spaces and $Z=X\times Y$. We observed in the
proof of Lemma~\ref{l:4.7} that there is a natural
embedding $\Aut X\rar \Aut Z$, which sends every $f\in \Aut X$ to
$f\times \operatorname{id}_Y$. In particular, if $G=\Aut I^A$,
then for any subset $B\sbs A$ there is a natural embedding $\Aut I^B\rar
G$. Let $G_B$ be the range of this embedding. The subgroup $G_B$ of $G$
consists precisely of self-homeomorphisms of $I^A$ which preserve the
$\a$-coordinate for every $\a\in A\setminus B$.

\begin{lemma}
\label{l:4.10} Let $A$ be a countable set and $G=\Aut I^A$. For
a subset
$B\sbs A$ let $G_B$, as above, be the image of $\Aut I^B$ under the natural
embedding $\Aut I^B\rar G$. Let $B_0\sbs B_1\sbs\dots$ be an increasing
sequence of infinite subsets of $A$ such that $\bigcup B_n=A$. Then for any
compact space $K$ the union of the subgroups $C(K, G_{B_n})$ is dense in
$C(K,G)$.
\end{lemma}

\begin{proof}
For every $n\in \omega$ pick a bijection $f_n: A\rar B_n$ so that every
$x\in A$ is fixed by $f_n$ for all but finitely many $n\in\omega$.
Let $r_n:G\rar G_{B_n}$ be the isomorphism induced by $f_n$,
and let $\tilde r_n:C(K, G)\rar C(K,G_{B_n})$ be the isomorphism induced by
$r_n$. It is easy to
see that the sequence $r_0, r_1,\dots$ of self-maps of $G$
converges to the identity map of $G$ uniformly on compact sets (use the fact
that compact subsets of $G$ are equicontinuous). It follows that the
sequence $\tilde r_0, \tilde r_1, \dots$ of self-maps of $C(K,G)$ converges
to the identity map pointwise, hence the union of the groups
$C(K, G_{B_n})=\tilde r_n(C(K,G))$ is dense in $C(K,G)$.
\end{proof}

For a
Hilbert space $H$ we denote by $U(H)$ the group of all unitary operators in
$H$, equipped with the strong operator topology (= the topology of pointwise
convergence).
As in Section~2, $\Sym E$ is the topological group of all permutations of a
discrete space $E$. Note that $w(U(H))=w(H)$, $w(\Aut I^\tau)=\tau$ and
$w(\Sym E)= \Card(E)$.
\begin{theorem} 
\label{t:4.11}
Let $G$ be one of the following topological groups:
\begin{itemize}
\item[(a)] the group $\Aut I^\tau$ of all self-homeomorphisms of a Hilbert
cube $I^\tau$.
\item[(b)] the group $\Sym E$ of all permutations of a set $E$;
\item[(c)] the unitary group $U(H)$ of a Hilbert space $H$;
\end{itemize}
Then:
\begin{enumerate}
\item the group $G$ has countable functional tightness if and only if its
weight is non-Ulam-measurable;
\item the group $G$ is $g$-sequential if and only if its weight is not a
$g$-sequential cardinal.
\end{enumerate}
\end{theorem} 

\begin{proof} We use the same method as in Theorem~\ref{t:4.4}, based on
Proposition~\ref{p:4.3}.
Let $\tau$ be
the weight of $G$ ($\tau=\Card(E)$ in the case (b) and $\tau=w(H)$ in the
case (c)), and let $A$ be a set of cardinality $\tau$. We construct a family
$D=\{d_\a:\a\in A\}$ of left-invariant pseudometrics on $G$ generating the topology of $G$
as follows. Let $f,g\in G$.

(a) If $G=\Aut I^\tau$, we assume that $A=\tau$. The group $G$ acts on the
right by isometries on the Banach space $B=C(I^\tau)$. 
For
$\a\in A$ let $b_\a\in B$ be the projection to the $\a$th coordinate,
$b_\a: I^A\rar I=[0,1]$. Let $d_\a(f,g)=\norm{b_\a f\obr-b_\a g\obr}$.

(b) If $G=\Sym E$, we may assume that $A=E$. For $x\in E$
let $d_x(f,g)=0$ if $f(x)=g(x)$ and $d_x(f,g)=1$ otherwise.

(c) If $G=U(H)$, let $\{e_\a:\a\in A\}$ be an orthonormal basis for $H$. For
$\a\in A$ let $d_\a(f,g)=\norm{f(e_\a)-g(e_\a)}$.

In all three cases the family $D=\{d_\a:\a\in A\}$ generates the topology
$\sT$ of $G$. Assume that $\tau$ is non-Ulam-measurable. We show that the
family $D$ satisfies the condition of Proposition~\ref{p:4.3}.

Let $f,g\in G$, and let $F$ be a finite set of free ultrafilters on $A$.
We must show that there is a sequence $g_0,g_1,\dots\in G$ converging to
$g$ such that for every $n\in \omega$ the set $\{\a\in A:d_\a(g_n,f)=0\}$
is in $p=\bigcap F$. Since the metrics $d_\a$ are left-invariant, we have
$d_\a(g_n,f)=d_\a(f^{-1}g_n,e)$, where $e$ is the unity of $G$. A sequence
$g_0,g_1,\dots$ converges to $g$ if and only if the sequence $f^{-1}g_0,
f^{-1}g_1,\dots$ converges to $f^{-1}g$. It follows from these remarks that
without loss of generality we may assume that $f=e$.

Consider first the case (a), when $G=\Aut I^\tau$.
Call a subset $B\sbs A$ {\it special\/} if $g$ is $I^B$-special in the sense
of Definition~\ref{4.6}.
In virtue of Shchepin's theorem \cite{Shch}, $A$ is covered by countable
special subsets. Since special subsets of $A$ are preserved by unions,
it follows that the set $A$ can be partitioned into countable infinite sets
$A_\ga$, $\ga<\tau$, so that each $C_\ga=\bigcup_{\beta<\ga} A_\ga$ is
special. Let $Y_\ga=I^{A_\ga}$ and
$X_\ga=I^{C_\ga}=\prod_{\beta<\ga} Y_\ga$.
Then $I^A$ is
identified with the product $Y=\prod_{\ga<\tau} Y_\ga$, and $g$ is a special
self-homeomorphism of $Y$ in the sense of Definition~\ref{4.8}.

Lemma~\ref{l:4.9} implies that $g$ belongs to a subgroup $P$ of $G$
isomorphic to an iterated semidirect product $\bigotimes_{\ga<\tau}
C(X_\ga, \Aut Y_\ga)$, thus $g$ can be identified with an element
$\{h_\ga\}$ of the product $\prod_{\ga<\tau} C(X_\ga, \Aut Y_\ga)$.
Since $\Card(A)$ is non-Ulam-measurable, there is a sequence
$B_0\supset B_1\dots\in\bigcap F$ such that $\bigcap B_n=\emptyset$. We may
also assume that $A_\ga\setminus B_0$ is infinite for every $\ga<\tau$.

It suffices to construct a sequence $g_0,g_1,\dots\in P$ such that
$\lim g_n=g$ and for every $n\in \omega$
the set $\{\a\in A: d_\a(g_n,e)=0\}$ contains $B_n$. The last condition
means that $g_n$ preserves the $\a$th coordinate $b_\a$ for every
$\a\in B_n$.     Given $\ga<\tau$ and $\a\in A_\ga$, let $H_\a\sbs\Aut Y_\a$
be the group of all self-homeomorphisms of $Y_\a$ which preserve the $\a$th
coordinate. Aplying Lemma~\ref{l:4.10} to the increasing sequence $\{A_\ga\setminus
B_n:n\in \omega\}$ of infinite subsets of $A_\ga$, we see
that for every $\ga<\tau$ there is a sequence
$h_{\ga,1},h_{\ga,2,}\dots\in C(X_\ga, \Aut Y_\ga)$ converging to $h_\ga$
such that
$h_{\ga,n}\in C(X_\ga, \bigcap_{\a\in B_n} H_\a)$
for every $n\in \omega$. Let $g_n=\{h_{\ga,n}:\ga<\tau\}$.
Then each $g_n$ can be regarded as an element of
$P=\bigotimes C(X_\ga, \Aut Y_\ga)\sbs G$, and the sequence $g_0, g_1,\dots$
is as required.

The cases (b) and (c) are similar but simpler. As above, let
$B_0\supset B_1\supset\dots\in p$ be a sequence of subsets of
$A$ such that $\bigcap B_n=\emptyset$. It suffices to construct a sequence
$g_0,g_1,\dots\in G$ such that $g=\lim g_n$ and $g_n(x)=x$ for every $x\in
B_n$ in the case (a) (in this case we assume, as above, that $A=E$) or
$g(e_\a)=e_\a$ for every $\a\in B_n$ in the case (b). In the case (a),
when $G=\Sym(E)$, the set $E$ can be partitioned into countable
$g$-invariant subsets $E_i$, $i\in I$. The product $\prod_i \Sym(E_i)$
can be identified with the subgroup $P$ of $G$ consisting of all $h\in G$
which leave each set $E_i$ invariant. Since $g\in P$, we can identify $g$
with an element of the product $\prod_i \Sym(E_i)$ and then construct the
sequence $g_0,g_1,\dots$ coordinate-wise.
The problem is thus
reduced to the case when $E$ is countable. Then $G$ is metrizable, and it
is clear that for any sequence $B_0\supset B_1\supset\dots$ of subsets of
$E$ such that $\bigcap B_n=\emptyset$ the subgroup
$$
S=\{h\in G: \text{
the set of fixed points of $h$ contains some }B_n\}
$$
is dense in $G$, since it contains the dense subgroup of all
permutations with finite support.
Hence there is a sequence $g_0,g_1,\dots\in S$ converging
to $g$, as required. In the case (b), when $G=U(H)$,
the set $A$ can be partitioned into countable subsets $A_i$, $i\in
I$, so that for every $i\in I$ the linear subspace $H_i$ of the
Hilbert space $H$, spanned by
the vectors $e_\a$, $\a\in A_i$, is $g$-invariant. The group $P=\prod_i
U(H_i)$, regarded as a subgroup of $G$, contains $g$. As above, this reduces
the problem to the case of separable $H$. The subgroup of all $h\in G$
which leave fixed all but finitely many vectors of the basis
$\{e_\a:\a\in A\}$ is dense in $G$. As above, this completes the argument.

If $\tau$ is measurable, the group $G$ is not $m$-stable (Section~3) in all
three cases (a), (b), (c),
since $G$ contains a copy of the cube $2^\tau$. Hence the identity map
of $G$ to its $m$-modification is $\omega$-continuous (Proposition~\ref{p:3.15}) but
not continuous. Thus $G$ does not have countable functional
tightness. Similarly,
if $\tau$ is $g$-sequential, then $G$ is not $g$-stable and moreover not
$g$-sequential.

Assume $\tau$ is not $g$-sequential. We must prove that every sequentially
continuous homomorphism $f:G\rar H$ is continuous.
Since we already know that $G$ has countable functional tightness,
it suffices to prove that $f$ is $\omega$-continuous.
We first consider the cases ~(b) and~(c).
We noted above
that every $g\in G$ is contained in a subgroup $P$ of $G$ which is
isomorphic to a product of separable metric groups. It is easy to see that
every countable subset of $G$ also is contained in such a subgroup $P$.
Theorem~\ref{t:3.2} implies that $P$ is $g$-sequential, hence the restriction of $f$
to $P$ is continuous. It follows that $f$ is $\omega$-continuous. In the
case (a), every countable subset of $G$ is contained in a subgroup $P$
which is an iterated semidirect product of metrizable groups of the form
$C(X, \Aut Y)$, where $X$ and $Y$ are compact and $Y$ is metrizable. Thus
it suffices to prove the following generalization of Theorem~\ref{t:3.2}:
if $\tau$ is not $g$-sequential, then any iterated semidirect product
$P=\bigotimes_{\a<\tau} H_\a$ of $g$-sequential groups $H_\a$ is
$g$-sequential. We noted after Definition~\ref{4.5} that
for any set $B$ of ordinals
$<\tau$ the iterated semidirect product $\bigotimes_{\a\in B}H_\a$ can be
regarded as a subgroup of $P$, hence the proof of Theorem~\ref{t:3.2} can be applied
without any changes.
\end{proof}

\section{Acknowledgement} 
The present paper was inspired by ideas of V.\,Pestov
who kindly sent to me 
the unpublished manuscript \cite{MNPS}.
I am indebted to V.\,Pestov for many helpful
discussions, and I regret that he declined my offer to be listed
as a coauthor of the paper.

\end{document}